\documentclass[reqno,10pt]{amsart}
\usepackage{amscd,amssymb,verbatim,array}
\usepackage{amsmath}
\usepackage[active]{srcltx} 

\numberwithin{equation}{section} \theoremstyle{plain}

\newtheorem{theorem}{Theorem}[section]
\newtheorem{lemma}[theorem]{Lemma}

\newtheorem{corollary}[theorem]{Corollary}
\newtheorem{definition}[theorem]{Definition}

\theoremstyle{definition}

\theoremstyle{remark}

\numberwithin{equation}{section}

\newcommand{\Det}{\operatorname{Det}}

\newcommand{\Span}{\operatorname{Span}}

\newcommand{\Fp}{\operatorname{Fp}}

\newcommand{\Dim}{\operatorname{dim}}

\newcommand{\DN}{\operatorname{DN}}

\newcommand{\Neu}{\operatorname{Neu}}
\newcommand{\Ker}{\operatorname{ker}}
\newcommand{\Spec}{\operatorname{Spec}}
\newcommand{\Tr}{\operatorname{Tr}}

\newcommand{\Res}{\operatorname{Res}}

\newcommand{\Dom}{\operatorname{Dom}}

\newcommand{\ddet}{\operatorname{det}}
\newcommand{\Id}{\operatorname{Id}}
\newcommand{\Mod}{\operatorname{Mod}}

\newcommand{\vol}{\operatorname{vol}}

\begin{document}

\title[The BFK-gluing formula for Robin boundary condition]
{The BFK type gluing formula of zeta-determinants for the Robin Boundary condition}
\author{Klaus Kirsten}
\address{GCAP-CASPER, Department of Mathematics, Baylor University, Waco, TX 76796, USA and Mathematical Reviews, American Mathematical Society,
 416 4$th$ Street, Ann Arbor, MI 48103, USA}

\email{Klaus\_Kirsten@Baylor.edu}

\author{Yoonweon Lee}

\address{Department of Mathematics Education, Inha University, Incheon, 22212, Korea and
School of Mathematics, Korea Institute for Advanced Study, 85 Hoegiro, Dongdaemun-gu, Seoul, 02455, Korea}

\email{yoonweon@inha.ac.kr}

\thanks{The first author was supported by the Baylor University Summer Sabbatical Programme. The second author was supported by the National Research Foundation of Korea with the Grant number 2016R1D1A1B01008091}

\begin{abstract}
In this paper we discuss the BFK type gluing formula for zeta-determinants of Laplacians with respect to the Robin boundary condition on a compact Riemannian manifold. As a special case, we discuss the gluing formula with respect to the Neumann boundary condition.
We also compute the difference of two zeta-determinants with respect to the Robin and Dirichlet boundary conditions.
We use this result to compute the zeta-determinant of a Laplacian on a cylinder when the Robin boundary condition is imposed, which extends a result in \cite{MKB}.
We also discuss the gluing formula more precisely when the product structure is given near a cutting hypersurface.
\end{abstract}

\maketitle

\section{Introduction}

The zeta-determinant of a Laplacian on a compact oriented Riemannian manifold with or without boundary is a global spectral invariant, which was introduced by Ray and Singer
in \cite{RS} to define the analytic torsion as an analytic counterpart of the Reidemeister torsion.
Since then, it has been one of important spectral invariants which plays a central role in geometry, topology and mathematical physics.
However, it is very hard to compute explicitly except only in a few cases.
Under these circumstances it is helpful in many ways to consider the gluing formula of the zeta-determinants.
The gluing formula for zeta-determinants of Laplacians was proved by Burghelea, Friedlander and Kappeler in \cite{BFK1} on a compact Riemannian manifold by using
the Dirichlet boundary condition (see also \cite{Ca} and \cite{Fo}), which we call the BFK-gluing formula.
In this case the Dirichlet-to-Neumann operator defined on a cutting hypersurface, which is a classical elliptic pseudodifferential operator  ($\Psi$DO) of order $1$, plays a central role.
By using complementary boundary condition (see Definition \ref{Definition:2.5} below), this result can be extended to the case of other boundary conditions, which is indicated in \cite{BFK1} and \cite{Fo}.

In this paper, we follow the arguments and methods in \cite{BFK1} to prove the gluing formula for zeta-determinants with respect to the Robin boundary condition,
which includes the Neumann boundary condition as a special case. We also discuss the difference of two zeta-determinants of Laplacians
with respect to the Robin and Dirichlet boundary conditions on a compact Riemannian manifold with boundary.
Comparing the gluing formula of the Robin boundary condition with the case of the Dirichlet boundary condition, there are two main differences.
First, in the Robin boundary condition the operator corresponding to the Dirichlet-to-Neumann operator is
a classical elliptic $\Psi$DO of order $-1$ so that it is a bounded operator. We use the inverse of this operator, which is a $\Psi$DO of order $1$, to define the zeta function and zeta-determinant.
Second, to prove the BFK-gluing formula of the zeta-determinant of a Laplace operator $\Delta$ in \cite{BFK1}, Burghelea et al. used a one parameter family $\Delta + \lambda$
for $\lambda \in {\mathbb C} - (-\infty, 0]$ rather than $\Delta$ itself, and then they used the derivative with respect to $\lambda$ with various functional analysis methods.
Under the Dirichlet boundary condition the Laplacian is a non-negative operator, and hence it is enough to take $\lambda$ to be non-negative real numbers.
However, under the Robin boundary condition the Laplacian may have finitely many negative eigenvalues.
In this case, we should take $\lambda = re^{i \theta}$ with $0 < \theta < \frac{\pi}{2}$ so that $\Delta + \lambda$ is a one parameter family of invertible operators for $r > 0$.

When a manifold has a product structure near a cutting hypersurface, the BFK-gluing formula is described more precisely.
Analogously we discuss the gluing formula for the Robin boundary condition and the difference of two zeta-determinants with respect to the Robin and Dirichlet boundary conditions  when a manifold has a product structure near a cutting hypersurface.
As an application, we compute the zeta-determinant of a Laplacian on a cylinder of length $L$ when the Robin boundary condition is imposed on  the boundary. This recovers a result in \cite{MKB}, which computes the zeta-determinant on a line segment when the Robin boundary condition is imposed.

\vspace{0.3 cm}

\section{The BFK-type gluing formula for the Robin boundary condition}

\vspace{0.2 cm}

Let $(M, g)$ be an $m$-dimensional compact oriented Riemannian manifold with boundary $\partial M$, where $\partial M$ is possibly empty.
We denote by ${\mathcal N}$ a closed hypersurface of $M$ such that
${\mathcal N} \cap \partial M = \emptyset$.
We also denote by $M_{0}$ and $g_{0}$ the closure of $M - {\mathcal N}$ and the induced metric so that $(M_{0}, g_{0})$ is a compact Riemannian manifolds with boundary
${\mathcal N}_{1} \cup {\mathcal N}_{2} \cup \partial M$, where ${\mathcal N}_{1} = {\mathcal N}_{2} = {\mathcal N}$.
We choose a unit normal vector field $\partial_{{\mathcal N}}$ along ${\mathcal N}$, which points outward to ${\mathcal N}_{1}$ and
inward to ${\mathcal N}_{2}$.
Let $\pi : {\mathcal E} \rightarrow M$ be a Hermitian vector bundle of rank $r_{0}$ and let ${\mathcal E}_{0}$ be the extension of ${\mathcal E}$ to $M_{0}$.
We denote by $(~ , ~ )$ the Hermitian inner product on ${\mathcal E}$. This together with the Riemannian metric $g$ gives an inner product on $L^{2}({\mathcal E})$,
which we denote by $\langle ~ , ~ \rangle$. We define an inner product on $L^{2}({\mathcal E}_{0})$ in a similar way.

We denote by $\Delta_{M}$ a Laplacian on $M$ acting on smooth sections of ${\mathcal E}$ with the principal symbol
$\sigma_{L}(\Delta_{M})(x^{\prime}, \xi^{\prime}) = \parallel \xi^{\prime} \parallel^{2}$, where $(x^{\prime}, \xi^{\prime})$ is a coordinate on $T^{\ast}M$.
It is well known \cite{BGV, Gi2} that there exist a connection
$\nabla : C^{\infty}({\mathcal E}) \rightarrow C^{\infty}(T^{\ast}M \otimes {\mathcal E})$ and
a bundle endomorphism $E : {\mathcal E} \rightarrow {\mathcal E}$ such that

\begin{eqnarray}   \label{E:2.1}
\Delta_{M} & = & - \left( \Tr \nabla^{2} + E \right).
\end{eqnarray}

\noindent
All through this paper, we assume that $\nabla$ is compatible with the Hermitian structure $(~, ~)$ and $E$ is self-adjoint, {\it i.e.}
for a smooth vector field $X$ on $M$ and smooth sections $s_{1}, s_{2} \in C^{\infty}({\mathcal E})$,

\begin{eqnarray}   \label{E:2.2}
X(s_{1}, s_{2}) = (\nabla_{X}s_{1}, s_{2}) + (s_{1}, \nabla_{X}s_{2}), \qquad (Es_{1}, s_{2}) = (s_{1}, Es_{2}).
\end{eqnarray}

\noindent
Then $\Delta_{M}$ is formally self-adjoint (Lemma 1.2.3 in \cite{Gi2}).
We also denote by $\Delta_{M_{0}}$ the extension of $\Delta_{M}$ to $M_{0}$.
If $\partial M \neq \emptyset$, we impose the Dirichlet or Neumann boundary condition $B$ on $\partial M$ and denote by $\Delta_{M, B}$ the Laplacian $\Delta_{M}$ subjected to $B$ on $\partial M$. We further assume that
$\Delta_{M, B}$ is a non-negative operator.
Many Laplacians including Hodge-De Rham Laplacians acting on differential forms satisfy these assumptions.
We next define the Robin boundary condition ${\frak R}_{i}$ on ${\mathcal N}_{i}$ in $M_{0}$ as follows.

\begin{definition}  \label{Definition:2.1}
Let $S : {\mathcal N} \rightarrow {\mathcal End} \left({\mathcal E}|_{{\mathcal N}} \right)$ be a section of the
endomorphism bundle ${\mathcal End} \left({\mathcal E}|_{{\mathcal N}} \right)$ over ${\mathcal N}$ such that
for each $y \in {\mathcal N}$, $S(y) : {\mathcal E}_{y} \rightarrow {\mathcal E}_{y}$ is symmetric.
Then, the Robin boundary conditions ${\frak R}_{1}$ and ${\frak R}_{2}$ are defined by
\begin{eqnarray*}
& & {\frak R}_{1} : C^{\infty}({\mathcal E}_{0}) \rightarrow C^{\infty}({\mathcal N}_{1}), \qquad
{\frak R}_{1} (\psi) = ~ \left( \nabla_{\partial_{{\mathcal N}}} \psi \right)\big|_{{\mathcal N}_{1}}
+ S \left( \psi|_{{\mathcal N}_{1}} \right),  \\
& & {\frak R}_{2} : C^{\infty}({\mathcal E}_{0}) \rightarrow C^{\infty}({\mathcal N}_{2}), \qquad
{\frak R}_{2} (\psi) = ~ \left(\nabla_{\partial_{{\mathcal N}}} \psi \right)\big|_{{\mathcal N}_{2}} + S \left( \psi|_{{\mathcal N}_{2}} \right).
\end{eqnarray*}
\end{definition}

\vspace{0.2 cm}
\noindent
{\it Remark}: If $S = 0$, then ${\frak R}_{1}$ and ${\frak R}_{2}$ are reduced to the Neumann boundary condition.  \newline

\vspace{0.2 cm}

All through this paper we are going to write $\Delta_{M_{0},{\frak R}_{0}, B} := \Delta_{M_{0},{\frak R}_{1},{\frak R}_{2}, B}$, $\Delta_{M_{0},D, N, B}$, etc, which means the Laplacian $\Delta_{M_{0}}$ with the boundary conditions on ${\mathcal N}_{1}$, ${\mathcal N}_{2}$ and $\partial M$, where ${\frak R}_{0} = ({\frak R}_{1}, {\frak R}_{2})$ and $D$, $N$ stand for the Dirichlet and Neumann conditions, respectively.
Then, the domains are given as follows.

\begin{eqnarray}    \label{E:2.3}
& & \Dom \left( \Delta_{M, B} \right) ~ = ~ \left\{ \phi \in C^{\infty}({\mathcal E}) \mid B(\phi) = 0 \right\},   \\
& & \Dom \left( \Delta_{M_{0}, {\frak R}_{0}, B} \right) ~ = ~ \left\{ \psi \in C^{\infty}({\mathcal E}_{0}) \mid {\frak R}_{1} \psi = {\frak R}_{2} \psi = 0, ~~ B(\psi) = 0 \right\}.  \nonumber
\end{eqnarray}

\vspace{0.2 cm}
\noindent
For $\phi$, $\psi \in \Dom \left( \Delta_{M_{0}, {\frak R}_{0}, B} \right)$ and
some local orthonormal basis $\{ e_{1}, \cdots, e_{m} \}$ on $TM$, the Green theorem (Section 1.4.8 in \cite{Gi2}) shows that

\begin{eqnarray}   \label{E:2.4}
\langle \Delta_{M_{0}, {\frak R}_{0}, B} \phi, ~ \psi \rangle_{M_{0}}  & = &
\int_{M_{0}} \bigg\{ \sum_{k=1}^{m} ( \nabla_{e_{k}} \phi, \nabla_{e_{k}} \psi) - (E \phi, \psi) \bigg\} dx ~ + ~
\int_{{\mathcal N}_{1}} ( S \phi|_{{\mathcal N}_{1}}, ~ \psi|_{{\mathcal N}_{1}} ) ~ dy   \\
& &  - ~ \int_{{\mathcal N}_{2}} ( S \phi|_{{\mathcal N}_{2}}, ~ \psi|_{{\mathcal N}_{2}} ) ~ dy ,    \nonumber
\end{eqnarray}

\noindent
which implies that $\Delta_{M_{0}, {\frak R}_{0}, B}$ is a symmetric operator having possibly finitely many negative eigenvalues.

\vspace{0.2 cm}
Let ${\mathcal D}$ be an elliptic $\Psi$DO of order $d > 0$ defined on an $L$-dimensional
compact Riemannian manifold with or without boundary which is bounded from below.
We assume that the principal symbol of ${\mathcal D}$ is positive definite.
When a boundary is not empty, ${\mathcal D}$ is assumed to be given a well posed boundary condition.
It is well known (Theorem 2.7 in \cite{GS}, (9) in \cite{Kas}) that $e^{-t {\mathcal D}}$ is a trace class operator and
$\Tr e^{-t {\mathcal D}}$ for $t \rightarrow 0^{+}$  has an asymptotic expansion of the type

\begin{eqnarray}    \label{E:2.11}
\Tr e^{-t {\mathcal D}} & = & \sum_{j=0}^{\infty} c_{j} t^{-\frac{L-j}{d}} + \sum_{k=1}^{\infty} \left( c^{\prime}_{k} t^{k} \ln t + c^{\prime\prime}_{k} t^{k} \right),
\end{eqnarray}

\noindent
where $c_{j}$'s and $c^{\prime}_{k}$'s are locally determined and $c_{k}^{\prime\prime}$'s are globally determined. In particular, a $t^{0} \ln t$ term does not appear.
If ${\mathcal D}$ is a differential operator, $c_{k}^{\prime}$ and $c_{k}^{\prime\prime}$ do not appear.
We denote the non-positive eigenspace of ${\mathcal D}$ by ${\mathcal L}$,
which is a finite dimensional vector space. We denote by ${\mathcal P}_{{\mathcal L}}$ the orthogonal projection onto ${\mathcal L}$.
We define

\begin{eqnarray}    \label{E:2.12}
\zeta_{{\mathcal D}}(s) & = & \sum_{0 \neq \mu_{j} \in \Spec({\mathcal D})} \mu_{j}^{-s} ~ = ~ \sum_{\mu_{j} < 0} \mu_{j}^{-s} ~ + ~
\frac{1}{\Gamma(s)} \int_{0}^{\infty} t^{s-1} \Tr e^{- t {\mathcal D}} \left( \Id  - {\mathcal P}_{{\mathcal L}} \right) dt.
\end{eqnarray}

\noindent
Then, (\ref{E:2.11}) shows that $\zeta_{{\mathcal D}}(s)$ is regular for $\Re s > \frac{L}{d}$ and has a meromorphic continuation having a regular value at $s = 0$.
When ${\mathcal D}$ is invertible, we define the zeta-determinant $\Det {\mathcal D}$  by

\begin{eqnarray}    \label{E:2.13}
\Det {\mathcal D} & = & e^{- \zeta_{{\mathcal D}}^{\prime}(0)}.
\end{eqnarray}

\noindent
When ${\mathcal D}$ has a non-trivial kernel, we define the modified zeta-determinant
$\Det^{\ast} {\mathcal D} := e^{- \zeta_{{\mathcal D}}^{\prime}(0)}$  by the same formula,
but we use the upper $\ast$ to distinguish it from the invertible case.

\vspace{0.2 cm}

In this paper, we are going to discuss the BFK-type gluing formula for $\Det^{\ast} \Delta_{M, B}$ with respect to the Robin boundary condition ${\frak R}_{i}$.
More precisely, when $\Delta_{M_{0}, {\frak R}_{0}, B}$ is a invertible operator, we are going to describe

\begin{eqnarray}   \label{E:2.14}
\ln \Det^{\ast} \Delta_{M, B} - \ln \Det \Delta_{M_{0}, {\frak R}_{0}, B}
\end{eqnarray}

\noindent
by using the zeta-determinant of an operator defined on $C^{\infty}({\mathcal E}|_{{\mathcal N}})$ and some additional terms defined locally on ${\mathcal N}$.
For this purpose, we are going to follow the arguments presented in \cite{BFK1}.
As suggested in \cite{BFK1}, we are going to consider one parameter families
$\Delta_{M, B} + \lambda$ and $~ \Delta_{M_{0}, {\frak R}_{0}, B} + \lambda ~$ rather than $\Delta_{M, B}$ and $\Delta_{M_{0}, {\frak R}_{0}, B}$.
Since $\Delta_{M_{0}, {\frak R}_{0}, B}$ may have negative eigenvalues, we cannot take $\lambda$ to be positive real numbers.
Instead, we take $\lambda$ as follows.

\begin{eqnarray}    \label{E:2.15}
\lambda = r e^{i \theta}, \qquad 0 \leq r < \infty, \qquad 0 < \theta < \frac{\pi}{2}.
\end{eqnarray}

\noindent
We also take the line $\{ re^{i \theta_{0}} \mid - \frac{3 \pi}{2} < \theta_{0} < - \pi \}$ as a branch cut to define the logarithm.
Then both $\Delta_{M, B} + \lambda$ and $~ \Delta_{M_{0}, {\frak R}_{0}, B} + \lambda ~$ are
invertible operators for $|\lambda| > 0$.
We first observe that for $0 \neq z = |z| e^{i \theta}$ with $0 < \theta < \frac{\pi}{2}$ and $\Re s > 0$,

\begin{eqnarray}     \label{E:2.16}
\int_{0}^{\infty} t^{s-1} e^{- t z} dt & = & z^{-s} \int_{0}^{\infty \cdot e^{i \theta}} t^{s-1} e^{-t} dt
~ = ~ z^{-s} \cdot \Gamma(s).
\end{eqnarray}

\vspace{0.2 cm}

We introduce the Poisson operator with respect to the Robin boundary conditions
${\frak R}_{i}$ as follows.

\begin{definition}   \label{Definition:2.2}
Suppose that $\Delta_{M_{0}, {\frak R}_{0}, B}$ is an invertible operator.
For $\lambda = r e^{i \theta}$ with $r \geq 0$ and $0 <  \theta < \frac{\pi}{2}$, the Poisson operator ${\mathcal P}_{0}(\lambda)$ with respect to ${\frak R}_{i}$ is a bounded linear map
${\mathcal P}_{0}(\lambda) : C^{\infty}({\mathcal E}|_{{\mathcal N}_{1}}) \oplus C^{\infty}({\mathcal E}|_{{\mathcal N}_{2}}) \rightarrow C^{\infty}({\mathcal E}_{0}) ~ $ satisfying
\begin{eqnarray*}
(\Delta_{M_{0}} + \lambda)  \cdot  {\mathcal P}_{0}(\lambda) = 0, \qquad
{\frak R}_{i}  \cdot  {\mathcal P}_{0}(\lambda) = \Id_{C^{\infty}({\mathcal E}|_{{\mathcal N}_{i}})},
\qquad B \cdot {\mathcal P}_{0}(\lambda) = 0.
\end{eqnarray*}
The invertibility of $~ \Delta_{M_{0}, {\frak R}_{0}, B} + \lambda ~$ shows that ${\mathcal P}_{0}(\lambda)$ is uniquely determined.
\end{definition}


\vspace{0.2 cm}
The Poisson operator ${\mathcal P}_{0}(\lambda)$ is constructed as follows.
For $x \in {\mathcal N}$, we denote by $\gamma_{x}(u)$ the unit speed geodesic starting from $x$
which points outward to $M_{1}$ and points inward to $M_{2}$ in the orthogonal way to ${\mathcal N}$.
Since ${\mathcal N}$ is compact, we can choose $\epsilon_{0} > 0$ uniformly such that $\gamma_{x}(u)$ is well defined
for all $x \in {\mathcal N}$ and $- \epsilon_{0} < u < \epsilon_{0}$.
In fact, $\gamma_{x}(u)$ is an integral curve to $\partial_{{\mathcal N}}$ on ${\mathcal N}$, and hence
we can extend $\partial_{{\mathcal N}}$ to a vector field on a collar neighborhood along $\gamma_{x}(u)$.
For an open subset $U$ of ${\mathcal N}$, we define

\begin{eqnarray}   \label{E:2.17}
V_{1} := (- \epsilon_{0},~ 0 ~] \times U, \qquad V_{2} ~ := ~  [~ 0, ~ \epsilon_{0}) \times U.
\end{eqnarray}

\noindent
Putting $x_{m} = u$, then $(x, x_{m})$ is a
local coordinate system in $V_{1} \cup V_{2}$, where $x = (x_{1}, \cdots, x_{m-1})$ is a local coordinate system in $U$.
When a normal coordinate system is chosen on $U$, we call $(x, x_{m})$ the boundary normal coordinate
system on $V_{1} \cup V_{2}$ (\cite{PS}).
For $f_{i} \in C^{\infty}({\mathcal E}|_{{\mathcal N}_{i}})$,
let ${\widetilde f}$ be an arbitrary extension of $f_{i}$  on $M_{0}$ such that
${\widetilde f}|_{{\mathcal N}_{i}} = f_{i}$ and $B({\widetilde f}) = 0$.
We choose a smooth function $h : M_{1} \cup M_{2} \rightarrow {\mathbb R}$ such that $h(x, x_{m}) = x_{m}$ on
$V_{1} \cup V_{2}$ and $0$ near $\partial M$.
Putting

\begin{eqnarray}   \label{E:2.18}
F(x, x_{m}) & = & h(x, x_{m}) ~ {\widetilde f}(x, x_{m}),
\end{eqnarray}

\noindent
we get

\begin{eqnarray}   \label{E:2.19}
{\frak R}_{i} \left( F(x, x_{m}) \right) ~ = ~ f_{i},  \qquad   B \left( F(x, x_{m}) \right) ~ = ~ 0.
\end{eqnarray}

\noindent
Then, ${\mathcal P}_{0}(\lambda) ( f_{1}, f_{2})$ is defined by

\begin{eqnarray}   \label{E:2.20}
{\mathcal P}_{0}(\lambda) (f_{1}, f_{2}) & = & F - (\Delta_{M_{0}, {\frak R}_{0}, B} + \lambda)^{-1} (\Delta_{M_{0}} + \lambda) F.
\end{eqnarray}

\noindent
We note that ${\mathcal P}_{0}(\lambda)$ satisfies the following properties.

\begin{lemma}  \label{Lemma:2.3}
\begin{eqnarray*}
\frac{d}{d \lambda} {\mathcal P}_{0}(\lambda) & = &
- (\Delta_{M_{0}, {\frak R}_{0}, B} + \lambda)^{-1} ~  \cdot ~ {\mathcal P}_{0}(\lambda).
\end{eqnarray*}
\end{lemma}

\begin{proof}
Differentiating the equalities in Definition \ref{Definition:2.2} with respect to $\lambda$, we obtain the following equalities.
\begin{eqnarray*}
{\mathcal P}_{0}(\lambda) ~ + ~
(\Delta_{M_{0}} + \lambda) \cdot \frac{d}{d \lambda} {\mathcal P}_{0}(\lambda) = 0,
\qquad  {\frak R}_{i}  \cdot \frac{d}{d \lambda} ~ {\mathcal P}_{0}(\lambda) = 0,
\qquad  B  \cdot \frac{d}{d \lambda} ~ {\mathcal P}_{0}(\lambda) = 0,
\end{eqnarray*}
which leads to the result.
\end{proof}

\noindent
We may identify $L^{2}({\mathcal E}_{0})$ with $L^{2}({\mathcal E})$ and hence we may regard
$(\Delta_{M_{0}, {\frak R}_{0}, B} + \lambda)^{-1}$ as a bounded
operator acting on $L^{2}({\mathcal E})$. This observation leads to the following fact.

\begin{lemma}   \label{Lemma:2.4}
\begin{eqnarray*}
(\Delta_{M, B} + \lambda)^{-1} - (\Delta_{M_{0}, {\frak R}_{0}, B} + \lambda)^{-1} & = &
{\mathcal P}_{0}(\lambda) \cdot \iota \cdot {\frak R}_{1} \cdot (\Delta_{M, B} + \lambda)^{-1},
\end{eqnarray*}
where $\iota : C^{\infty}({\mathcal E}|_{{\mathcal N}}) \rightarrow C^{\infty}({\mathcal E}|_{{\mathcal N}_{1}}) \oplus C^{\infty}({\mathcal E}|_{{\mathcal N}_{2}})$ is the diagonal inclusion.
\end{lemma}

\begin{proof}
Setting $A(\lambda) = (\Delta_{M, B} + \lambda)^{-1} - (\Delta_{M_{0}, {\frak R}_{0}, B} + \lambda)^{-1}$, we get
\begin{eqnarray*}
(\Delta_{M} + \lambda) \cdot  A(\lambda) = 0,  \qquad  B \cdot A(\lambda) = 0, \qquad
{\frak R}_{0} \cdot A(\lambda) = \iota \cdot {\frak R}_{1} \cdot (\Delta_{M, B} + \lambda)^{-1},
\end{eqnarray*}
which leads to the result.
\end{proof}

We next define a trace map $\gamma_{0}$ and a difference operator $\delta_{d}$ as follows.
\begin{eqnarray}    \label{E:2.21}
& & \gamma_{0} : C^{\infty}({\mathcal E}_{0}) \rightarrow  C^{\infty}({\mathcal E}|_{{\mathcal N}_{1}}) \oplus C^{\infty}({\mathcal E}|_{{\mathcal N}_{2}}),  \qquad
\gamma_{0}(\psi) = (\psi|_{{\mathcal N}_{1}}, ~ \psi|_{{\mathcal N}_{2}}),   \\
& & \delta_{d} : C^{\infty}({\mathcal E}|_{{\mathcal N}_{1}}) \oplus C^{\infty}({\mathcal E}|_{{\mathcal N}_{2}}) \rightarrow
C^{\infty}({\mathcal E}|_{{\mathcal N}}), \qquad
\delta_{d}(f, ~ g) = f - g.      \nonumber
\end{eqnarray}

\noindent
We define the complementary boundary conditions as follows (\cite{BFK1}).

\begin{definition}   \label{Definition:2.5}
Let ${\frak B}_{1}$ and ${\frak B}_{2}$ be two elliptic boundary conditions for $\Delta_{M_{0}, B}$ on ${\mathcal N}_{1} \cup {\mathcal N}_{2}$.
${\frak B}_{1}$ and ${\frak B}_{2}$ are called complementary if the null space of
\begin{eqnarray*}
{\frak B}_{1} \oplus {\frak B}_{2} : C^{\infty}({\mathcal E}_{0}) \rightarrow \bigg( C^{\infty}({\mathcal E}|_{{\mathcal N}_{1}}) \oplus C^{\infty}({\mathcal E}|_{{\mathcal N}_{2}}) \bigg) \oplus
\bigg( C^{\infty}({\mathcal E}|_{{\mathcal N}_{1}}) \oplus C^{\infty}({\mathcal E}|_{{\mathcal N}_{2}}) \bigg)
\end{eqnarray*}
consists of sections $\psi \in C^{\infty}({\mathcal E}_{0})$ satisfying $\psi\big|_{{\mathcal N}_{1} \cup {\mathcal N}_{2}}  ~ = ~ (\nabla_{\partial_{{\mathcal N}}} \psi)\big|_{{\mathcal N}_{1} \cup {\mathcal N}_{2}} ~ = ~ 0$.
\end{definition}

\noindent
For an example, the Dirichlet condition is complementary to the Neumann condition, and vice versa.
Similarly, the Dirichlet condition is also complementary to ${\frak R}_{0}$.
We use this fact to define an operator, which corresponds to the Dirichlet-to-Neumann operator in \cite{BFK1}.

\begin{definition}   \label{Definition:2.6}
We assume that $\Delta_{M_{0}, {\frak R}_{0}, B}$ is an invertible operator. For $\lambda = r e^{i\theta}$ with
$r \geq 0$, we define an operator $R_{S}(\lambda) : C^{\infty}({\mathcal E}|_{{\mathcal N}}) \rightarrow C^{\infty}({\mathcal E}|_{{\mathcal N}})$ as follows.
For $f \in C^{\infty}({\mathcal E}|_{{\mathcal N}})$,
\begin{eqnarray*}
R_{S}(\lambda) ~ f & = & \delta_{d} \cdot \gamma_{0} \cdot {\mathcal P}_{0}(\lambda) \cdot \iota ~(f).
\end{eqnarray*}
\end{definition}

\vspace{0.2 cm}

For $\lambda \in {\mathbb C} - (-\infty, 0)$, we define ${\mathcal W}_{S}(\lambda) : C^{\infty}({\mathcal E}|_{{\mathcal N}_{1}}) \oplus C^{\infty}({\mathcal E}|_{{\mathcal N}_{2}}) \rightarrow C^{\infty}({\mathcal E}|_{{\mathcal N}_{1}}) \oplus C^{\infty}({\mathcal E}|_{{\mathcal N}_{2}})$ as follows.
Given $(f, g) \in C^{\infty}({\mathcal E}|_{{\mathcal N}_{1}}) \oplus C^{\infty}({\mathcal E}|_{{\mathcal N}_{2}})$, we choose $\psi \in C^{\infty}({\mathcal E}_{0})$ satisfying

\begin{eqnarray}   \label{E:2.70}
(\Delta_{M_{0}} + \lambda) \psi = 0, \qquad \psi|_{{\mathcal N}_{1}} = f, \qquad \psi|_{{\mathcal N}_{2}} = g,
\qquad B(\psi) = 0.
\end{eqnarray}

\noindent
Indeed, $\psi$ is constructed as follows. Let ${\widetilde f}$ be any smooth extension of $(f, g)$ on $M_{0}$
satisfying ${\widetilde f}|_{{\mathcal N}_{1}} = f$, ${\widetilde f}|_{{\mathcal N}_{2}} = g$
and $B({\widetilde f}) = 0$. Then, $\psi$ is obtained by
$\psi = {\widetilde f} - (\Delta_{M_{0}, D, D, B} + \lambda)^{-1} (\Delta_{M_{0}} + \lambda) {\widetilde f}$.
We define

\begin{eqnarray}    \label{E:2.71}
{\mathcal W}_{S}(\lambda)(f, g) ~ = ~
\left( \big( \nabla_{\partial_{{\mathcal N}}} \psi \big)|_{{\mathcal N}_{1}} + S f, ~
\big( \nabla_{\partial_{{\mathcal N}}} \psi \big)|_{{\mathcal N}_{2}} + S g \right).
\end{eqnarray}

\noindent
Then, it follows that

\begin{eqnarray}    \label{E:2.72}
\Ker {\mathcal W}_{S}(0) & = &  \big\{ (\psi|_{{\mathcal N}_{1}}, ~\psi|_{{\mathcal N}_{2}}) \mid
\psi \in \Ker \Delta_{M_{0}, {\frak R}_{0}, B} \big\}.
\end{eqnarray}

\noindent
Putting $\psi = {\mathcal P}_{0}(\lambda) \iota(f)$, we get

\begin{eqnarray}   \label{E:2.22}
(f, ~ f) ~ = ~ \big( \left( \nabla_{\partial_{{\mathcal N}}} \psi \right)|_{{\mathcal N}_{1}} + S \left( \psi|_{{\mathcal N}_{1}} \right), ~ \left( \nabla_{\partial_{{\mathcal N}}} \psi \right)|_{{\mathcal N}_{2}} +
S \left( \psi|_{{\mathcal N}_{2}} \right) \big) ~ = ~
{\mathcal W}_{S}(\lambda) \big( \psi|_{{\mathcal N}_{1}}, ~ \psi|_{{\mathcal N}_{2}} \big),
\end{eqnarray}

\noindent
which shows that

\begin{eqnarray}  \label{E:2.90}
R_{S}(\lambda) & = & \delta_{d} \cdot {\mathcal W}_{S}(\lambda)^{-1} \cdot \iota.
\end{eqnarray}

\vspace{0.2 cm}
To describe ${\mathcal W}(\lambda)$ precisely,
we introduce $Q_{i}(\lambda) : C^{\infty}({\mathcal E}|_{{\mathcal N}_{i}}) \rightarrow C^{\infty}({\mathcal E}|_{{\mathcal N}_{i}})$ and $V_{i}(\lambda) : C^{\infty}({\mathcal E}|_{{\mathcal N}_{i}}) \rightarrow C^{\infty}({\mathcal E}|_{{\mathcal N}_{j}})$ as follows,
where $j = i + (-1)^{i+1}$.
For $ f \in C^{\infty}({\mathcal E}|_{{\mathcal N}})$, choose $\phi$, $\psi \in C^{\infty}({\mathcal E}_{0})$ such that

\begin{eqnarray}   \label{E:2.5}
& & \left( \Delta_{M_{0}} + \lambda \right) \phi ~ = ~ \left( \Delta_{M_{0}} + \lambda \right) \psi = 0, \qquad B(\phi) = B(\psi) = 0,  \\
& & \phi|_{{\mathcal N}_{1}} = f, \quad \phi|_{{\mathcal N}_{2}} = 0,  \qquad
\psi|_{{\mathcal N}_{1}} = 0, \quad \psi|_{{\mathcal N}_{2}} = f.    \nonumber
\end{eqnarray}

\noindent
We define

\begin{eqnarray}   \label{E:2.6}
& & Q_{1}(\lambda) f = \left( \nabla_{\partial_{{\mathcal N}}} \phi \right)|_{{\mathcal N}_{1}}, \qquad
V_{1}(\lambda)f = - \left(\nabla_{\partial_{{\mathcal N}}} \phi \right)|_{{\mathcal N}_{2}},  \\
& & Q_{2}(\lambda) f = - \left( \nabla_{\partial_{{\mathcal N}}} \psi \right)|_{{\mathcal N}_{2}}, \qquad
V_{2}(\lambda)f = \left( \nabla_{\partial_{{\mathcal N}}} \psi \right)|_{{\mathcal N}_{1}}.   \nonumber
\end{eqnarray}

\noindent
Then, $Q_{i}(\lambda)$'s are elliptic pseudodifferential operators of order $1$ and $V_{i}(\lambda)$'s are proved to be smoothing operators in Lemma 2.3 in \cite{KL4}.
When $\lambda = 0$, the kernels are given by

\begin{eqnarray}  \label{E:2.7}
& & \Ker Q_{1}(0) = \{ \phi|_{{\mathcal N}_{1}} \mid \phi \in \Ker \Delta_{M_{0}, N, D, B} \}, \qquad
\Ker Q_{2}(0) = \{ \phi|_{{\mathcal N}_{2}} \mid \phi \in \Ker \Delta_{M_{0}, D, N, B} \},  \\
& & \Ker \left( Q_{1}(0) + S \right) = \{ \phi|_{{\mathcal N}_{1}} \mid \phi \in \Ker \Delta_{M_{0}, {\frak R}_{1},D,B} \}, \quad
\Ker \left( Q_{2}(0) - S \right) = \{ \phi|_{{\mathcal N}_{2}} \mid \phi \in \Ker \Delta_{M_{0},D,{\frak R}_{2}, B} \},  \nonumber
\end{eqnarray}

\noindent
which leads to

\begin{eqnarray}   \label{E:2.8}
& & \Dim \Ker Q_{1}(0) ~ = ~ \Dim \Ker \Delta_{M_{0}, N, D, B}, \qquad
\Dim \Ker Q_{2}(0) ~ = ~ \Dim \Ker \Delta_{M_{0}, D, N, B}, \\
& & \Dim \Ker \left( Q_{1}(0) + S \right) ~ = ~ \Ker \Delta_{M_{0}, {\frak R}_{1}, D, B}, \qquad
\Dim \Ker \left( Q_{2}(0) - S \right) ~ = ~ \Ker \Delta_{M_{0}, D, {\frak R}_{2}, B}.   \nonumber
\end{eqnarray}

\noindent
For $f$, $h \in C^{\infty}({\mathcal E}|_{{\mathcal N}})$ and $i = 1, 2$, choose $\phi_{i}$, $\psi_{i} \in C^{\infty}({\mathcal E}_{0})$ such that

\begin{eqnarray*}
& & \left( \Delta_{M_{0}} + \lambda \right) \phi_{i} ~ = ~ \left( \Delta_{M_{0}} + \lambda \right) \psi_{i} = 0, \qquad B(\phi_{i}) = B(\psi_{i}) = 0,  \\
& & \phi_{i}|_{{\mathcal N}_{i}} = f, \qquad \psi_{i}|_{{\mathcal N}_{i}} = h, \qquad
\phi_{i}|_{{\mathcal N}_{j}} = \psi_{i}|_{{\mathcal N}_{j}} = 0,
\end{eqnarray*}

\noindent
where $j = i + (-1)^{i+1}$.
The Green theorem shows that

\begin{eqnarray}    \label{E:2.9}
\langle Q_{i}(\lambda) f, ~ h \rangle_{{\mathcal N}_{i}} & = &
\int_{M_{0}} \bigg\{ \sum_{k=1}^{m} ( \nabla_{e_{k}} \phi_{i},  \nabla_{e_{k}} \psi_{i}) - (E \phi_{i},  \psi_{i}) \bigg\} dx ~ + ~
\lambda ~ \int_{M_{0}} ( \phi_{i},  \psi_{i} ) dx,
\end{eqnarray}

\noindent
which implies that the adjoint of $Q_{i}(\lambda)$ is $Q_{i}({\overline \lambda})$ and
$Q_{i}(\lambda)$ is bounded from below for $\lambda \in {\mathbb R}$.
This shows that $Q_{i}(\lambda) + S$ is also bounded below.
For later use, we put

\begin{eqnarray}    \label{E:2.10}
R_{\DN}(\lambda) & := & Q_{1}(\lambda) + Q_{2}(\lambda) + V_{1}(\lambda) + V_{2}(\lambda)
~ = ~ \delta_{d} \cdot {\mathcal W}(\lambda) \cdot \iota,
\end{eqnarray}

\noindent
where ${\mathcal W}(\lambda) = {\mathcal W}_{S}(\lambda)$ with $S = 0$.
$R_{\DN}(\lambda)$ is called the Dirichlet-to-Neumann operator and plays a central role in the BFK-gluing formula with respect to the Dirichlet boundary condition
(\cite{BFK1}, \cite{Ca}, \cite{KL1}, \cite{KL2}, \cite{KL3}, \cite{KL4}, \cite{Le1}, \cite{Le2}).
It follows from the definition that

\begin{eqnarray*}
{\mathcal W}_{S}(\lambda)  \left( \begin{array}{clcr} f \\ g \end{array} \right) & = & \left( \begin{array}{clcr} Q_{1}(\lambda) + S & V_{2}(\lambda) \\ - V_{1}(\lambda) &
- Q_{2}(\lambda) + S \end{array} \right) \left( \begin{array}{clcr} f \\ g \end{array} \right),
\end{eqnarray*}

\noindent
which together with (\ref{E:2.90}) leads to the following result, which is useful in computing the homogeneoue symbol of $R_{S}(\lambda)$.

\begin{lemma}   \label{Lemma:2.7}
\begin{eqnarray*}
R_{S}(\lambda) & = & \delta_{d} \cdot {\mathcal W}_{S}(\lambda)^{-1} \cdot \iota ~ = ~
\delta_{d} \cdot \left( \begin{array}{clcr} Q_{1}(\lambda) + S & V_{2}(\lambda) \\ - V_{1}(\lambda) & - Q_{2}(\lambda) + S \end{array} \right)^{-1} \cdot \iota.
\end{eqnarray*}
In particular, the homogeneous symbol of $R_{S}(\lambda)$ is given by

\begin{eqnarray*}
\sigma(R_{S}(\lambda)) & = & \sigma\big((Q_{1}(\lambda) + S)^{-1} \big) ~ + ~
\sigma \big((Q_{2}(\lambda) - S)^{-1}\big).
\end{eqnarray*}
If $M_{0}$ is not connected, {\it i.e.} $M = M_{1} \cup_{{\mathcal N}} M_{2}$, then $V_{1}(\lambda) = V_{2}(\lambda) = 0$ and hence
\begin{eqnarray*}
R_{S}(\lambda) & = & (Q_{1}(\lambda) + S)^{-1} ~ + ~ (Q_{2}(\lambda) - S)^{-1}.
\end{eqnarray*}
\end{lemma}

\vspace{0.2 cm}

Under some fixed coordinate system the homogeneous symbol of $R_{S}(\lambda)$ can be computed in terms of those of $Q_{1}(\lambda)$ and $Q_{2}(\lambda)$
by using the formula of the composition and inverse of pseudodifferential operators (\cite{Gi1}, \cite{Sh}).
If $\Delta_{M_{0}, {\frak R}_{0}, B}$ is an invertible operator, then $R_{S}(0)$ is well defined,
but may have a non-trivial kernel.

\begin{lemma}   \label{Lemma:2.8}
Suppose that $\Delta_{M_{0}, {\frak R}_{0}, B}$ is an invertible operator.
The kernel of $R_{S}(0)$ is given by
\begin{eqnarray*}
\Ker R_{S}(0) & = & \{ (\nabla_{\partial_{{\mathcal N}}} \psi )\big|_{{\mathcal N}} + S (\psi|_{{\mathcal N}}) ~ \big| ~ \psi \in \Ker \Delta_{M, B} \}.
\end{eqnarray*}
In particular, $\Dim \Ker R_{S}(0) ~ = ~ \Dim \Ker R_{\DN}(0) ~ = ~ \Dim \Ker \Delta_{M, B}$.
\end{lemma}

\begin{proof}
Setting $f = (\nabla_{\partial_{{\mathcal N}}} \psi )\big|_{{\mathcal N}} + S (\psi|_{{\mathcal N}})~$ for $\psi \in \Ker \Delta_{M, B}$,
then $\psi$ is equal to ${\mathcal P}_{0}(0)(\iota(f))$, which is a smooth section on $M$.
Hence $f$ belongs to $\Ker R_{S}(0)$.
Let $f \in \Ker R_{S}(0)$. Then,
\begin{eqnarray*}
\bigg( \nabla_{\partial_{{\mathcal N}}}{\mathcal P}_{0}(0)(f, f) \bigg)\big|_{{\mathcal N}_{1}} + S \bigg( {\mathcal P}_{0}(0)(f, f) \big|_{{\mathcal N}_{1}} \bigg) & = &
\bigg( \nabla_{\partial_{{\mathcal N}}}{\mathcal P}_{0}(0)(f, f) \bigg)\big|_{{\mathcal N}_{2}} + S \bigg( {\mathcal P}_{0}(0)(f, f) \big|_{{\mathcal N}_{2}} \bigg) ~ = ~ f,  \\
{\mathcal P}_{0}(0) (f, f)\big|_{{\mathcal N}_{1}} & = & {\mathcal P}_{0}(0) (f, f)\big|_{{\mathcal N}_{2}},
\end{eqnarray*}

\noindent
which shows that  $ \bigg(\nabla_{\partial_{{\mathcal N}}}{\mathcal P}_{0}(0)(f, f) \bigg)\big|_{{\mathcal N}_{1}} = \bigg( \nabla_{\partial_{{\mathcal N}}}{\mathcal P}_{0}(0)(f, f) \bigg)\big|_{{\mathcal N}_{2}}  ~$.
Hence $~ {\mathcal P}_{0}(0) (f, f) ~$ is continuous and smooth on $M$, which
shows that it belongs to $\Ker \Delta_{M, B}$.
\end{proof}

\vspace{0.2 cm}
\noindent{\it Remark} : It is interesting to compare the kernel of $R_{S}(0)$ with that of $R_{\DN}(0)$, which is
\begin{eqnarray}   \label{E:2.25}
\Ker R_{\DN}(0) & = & \{  \psi \big|_{{\mathcal N}} ~ \big| ~ \psi \in \Ker \Delta_{M, B} \}.
\end{eqnarray}

\vspace{0.2 cm}

Since $R_{S}(\lambda)^{-1}$ is a $\Psi$DO of order $1$, the zeta function and zeta-determinant of $R_{S}(\lambda)^{-1}$ are well defined.
We define the zeta function $\zeta_{R_{S}(\lambda)}(s)$ associated to $R_{S}(\lambda)$ by

\begin{eqnarray}    \label{E:2.26}
\zeta_{R_{S}(\lambda)}(s) & := & \zeta_{R_{S}(\lambda)^{-1}}(-s).
\end{eqnarray}

\noindent
Then, $\zeta_{R_{S}(\lambda)}(s)$ is analytic for $\Re s < - (m-1)$ and has an meromorphic continuation having a regular value at zero.

\begin{definition}   \label{Definition:2.9}
We define the zeta-determinant of $R_{S}(\lambda)$ by
\begin{eqnarray*}
\ln \Det R_{S}(\lambda) & := & \zeta_{R_{S}(\lambda)^{-1}}^{\prime}(0) ~ = ~ - \ln \Det R_{S}(\lambda)^{-1}.
\end{eqnarray*}
\end{definition}

\vspace{0.2 cm}

\noindent
From Lemma \ref{Lemma:2.3} and Lemma \ref{Lemma:2.4}, we get the following equalities.

\begin{eqnarray}   \label{E:2.27}
\frac{d}{d \lambda} R_{S}(\lambda) & = & \delta_{d} \cdot \gamma_{0} \cdot \frac{d}{d \lambda} {\mathcal P}_{0}(\lambda) \cdot \iota  \\
& = & \delta_{d} \cdot \gamma_{0} \cdot \left( - ~ (\Delta_{M_{0}, {\frak R}_{0}, B} + \lambda)^{-1} ~  \cdot ~ {\mathcal P}_{0}(\lambda) \right) \cdot \iota
\nonumber  \\
& = & \delta_{d} \cdot \gamma_{0} \cdot \left\{ {\mathcal P}_{0}(\lambda) \cdot \iota \cdot {\frak R}_{1} \cdot (\Delta_{M, B} + \lambda)^{-1} ~ - ~
(\Delta_{M, B} + \lambda)^{-1} \right\} ~  \cdot ~ {\mathcal P}_{0}(\lambda) \cdot \iota.   \nonumber
\end{eqnarray}

\begin{lemma}   \label{Lemma:2.10}
\begin{eqnarray*}
\delta_{d} \cdot \gamma_{0} \cdot (\Delta_{M, B} + \lambda)^{-1} ~  \cdot ~ {\mathcal P}_{0}(\lambda) \cdot \iota & = & 0.
\end{eqnarray*}
\end{lemma}

\begin{proof}
Definition \ref{Definition:2.2} shows that for $f \in C^{\infty}({\mathcal E}|_{{\mathcal N}})$,  $~ {\mathcal P}_{0}(\lambda) (f, f)$ belongs to $L^{2}({\mathcal E})$,
which shows that $~(\Delta_{M, B} + \lambda)^{-1} \cdot {\mathcal P}_{0}(\lambda)  (f, f) \in H^{2}({\mathcal E}) ~$.
Hence, the projection onto ${\mathcal N}$ is well defined (cf. Theorem 11.4 in \cite{BW}), and
we get
\begin{eqnarray*}
\delta_{d} \cdot \gamma_{0} \cdot (\Delta_{M, B} + \lambda)^{-1} \cdot {\mathcal P}_{0}(\lambda)  (f, f) ~ = ~ 0,
\end{eqnarray*}
which completes the proof of the lemma.
\end{proof}

\noindent
The above lemma yields the following equalities.

\begin{eqnarray}   \label{E:2.28}
\frac{d}{d \lambda} R_{S}(\lambda) & = & \delta_{d} \cdot \gamma_{0} \cdot {\mathcal P}_{0}(\lambda) \cdot \iota \cdot {\frak R}_{1} \cdot (\Delta_{M, B} + \lambda)^{-1} ~
\cdot ~ {\mathcal P}_{0}(\lambda) \cdot \iota  \\
& = & R_{S}(\lambda) \cdot {\frak R}_{1} \cdot (\Delta_{M, B} + \lambda)^{-1} ~
\cdot ~ {\mathcal P}_{0}(\lambda) \cdot \iota,     \nonumber
\end{eqnarray}

\noindent
which shows that

\begin{eqnarray}   \label{E:2.29}
R_{S}(\lambda)^{-1} \cdot \frac{d}{d \lambda} R_{S}(\lambda) & = &
{\frak R}_{1} \cdot (\Delta_{M, B} + \lambda)^{-1} ~
\cdot ~ {\mathcal P}_{0}(\lambda) \cdot \iota.
\end{eqnarray}

\vspace{0.2 cm}
\noindent
We put $\nu := \big[ \frac{m-1}{2} \big] + 1$ and note the following equalities.

\begin{eqnarray}     \label{E:2.30}
& & \frac{d^{\nu}}{d \lambda^{\nu}} \left\{ \ln \Det (\Delta_{M, B} + \lambda) - \ln \Det (\Delta_{M_{0}, {\frak R}_{0}, B} + \lambda) \right\}      \\
& = & \Tr \left\{ \frac{d^{\nu-1}}{d \lambda^{\nu-1}} \left( (\Delta_{M,B} + \lambda)^{-1} -
(\Delta_{M_{0}, {\frak R}_{0},B} + \lambda)^{-1} \right) \right\}
 =  \Tr \left\{ \frac{d^{\nu-1}}{d \lambda^{\nu-1}} \left( {\mathcal P}_{0}(\lambda) \cdot \iota \cdot {\frak R}_{1} \cdot (\Delta_{M,B} + \lambda)^{-1} \right) \right\}    \nonumber \\
& = & \Tr \left\{ \frac{d^{\nu-1}}{d \lambda^{\nu-1}} \left\{
{\frak R}_{1} \cdot (\Delta_{M} + \lambda)^{-1} \cdot {\mathcal P}_{0}(\lambda) \cdot \iota \right\} \right\}
 =  \Tr \left\{ \frac{d^{\nu-1}}{d \lambda^{\nu-1}} \left(
R_{S}(\lambda)^{-1} \cdot \frac{d}{d \lambda} R_{S}(\lambda) \right) \right\}    \nonumber \\
& = & \frac{d^{\nu}}{d \lambda^{\nu}} \ln \Det R_{S}(\lambda),   \nonumber
\end{eqnarray}

\noindent
which leads to the following result.

\begin{theorem}    \label{Theorem:2.11}
Let $(M, g)$ be a compact oriented Riemannian manifold with boundary $\partial M$ and let ${\mathcal N}$ be a closed hypersurface such that ${\mathcal N} \cap \partial M = \emptyset$. We denote by $\Delta_{M, B}$ and $\Delta_{M_{0}, {\frak R}_{0}, B}$ the Laplacians on $M$ and $M_{0}$ with the boundary conditions $B$ on $\partial M$ and ${\frak R}_{i}$ on ${\mathcal N}_{i}$.
Then, for
$\lambda = r e^{i \theta}$ with $0 < r < \infty$ and $0 < \theta < \frac{\pi}{2}$,
there exists a polynomial $~ P(\lambda) = \sum_{j=0}^{\big[ \frac{m-1}{2} \big]} a_{j} \lambda^{j} ~$ such that

\begin{eqnarray*}
\ln \Det (\Delta_{M, B} + \lambda) ~ - ~ \ln \Det (\Delta_{M_{0}, {\frak R}_{0}, B} + \lambda)
 & = & \sum_{j=0}^{\big[ \frac{m-1}{2} \big]} a_{j} \lambda^{j} + \ln \Det R_{S}(\lambda).
\end{eqnarray*}
\end{theorem}

\vspace{0.2 cm}

To determine the polynomial $P(\lambda)$, we use the asymptotic expansions for $|\lambda| \rightarrow \infty$.
We first note that each term in the left hand side of Theorem \ref{Theorem:2.11} has an asymptotic expansion for $|\lambda| \rightarrow \infty$ whose constant term is zero
(cf. Lemma 2.1 in \cite{KL2}, \cite{Vo}). We also note that $R_{S}(\lambda)$ is an elliptic $\Psi$DO with parameter of weight $2$. We refer to the Appendix of \cite{BFK1} or
\cite{Sh} for the definition of elliptic operators with parameter.
It is shown in the Appendix of \cite{BFK1} that
$\ln \Det R_{S}(\lambda)$ has an asymptotic expansion for $|\lambda| \rightarrow \infty$ of the following type:

\begin{eqnarray}   \label{E:2.31}
\ln \Det R_{S}(\lambda) ~ = ~ - \ln \Det \left( R_{S}(\lambda)^{-1} \right)
& \sim & \sum_{j=0}^{\infty} \pi_{j} |\lambda|^{\frac{m-1-j}{2}} + \sum_{j=0}^{m-1} q_{j} |\lambda|^{\frac{m-1-j}{2}} \ln |\lambda|,
\end{eqnarray}

\noindent
where the coefficients $\pi_{j}$ and $q_{j}$ are computed by integration of some local densities obtained from the homogeneous symbol of
the resolvent of $R_{S}(\lambda)^{-1}$.
This observation shows that $a_{0} = - \pi_{m-1}$, where $\pi_{m-1}$ is the constant term in the asymptotic expansion of
$\ln \Det R_{S}(\lambda)$.

We next discuss the coefficient $\pi_{m-1}$.
We denote the homogeneous symbols of $R_{S}(\lambda)^{-1}$ and its resolvent $\left(\mu - R_{S}(\lambda)^{-1} \right)^{-1}$ as follows (\cite{Gi1}, \cite{Sh}).

\begin{eqnarray}    \label{E:2.32}
& & \sigma \left( R_{S}(\lambda)^{-1} \right) ~ \sim ~ \sum_{j=0}^{\infty} \theta_{1-j}(x, \xi, \lambda, \mu),  \\
& & \sigma \left( \left( \mu -  R_{S}(\lambda)^{-1}  \right)^{-1} \right)(x, \xi, \lambda, \mu) ~ \sim ~ \sum_{j=0}^{\infty} r_{-1-j}(x, \xi, \lambda, \mu),   \nonumber
\end{eqnarray}
\noindent
where

\begin{eqnarray}   \label{E:2.33}
& & r_{-1}(x, \xi, \lambda, \mu) ~ = ~ \left( \mu - \theta_{1} \right)^{-1},  \\
& & r_{-1-j}(x, \xi, \lambda, \mu) ~ = ~ \left( \mu - \theta_{1} \right)^{-1}\,\, \sum_{k=0}^{j-1} \,\,\sum_{|\omega|+l+k=j} \frac{1}{\omega!}
\partial_{\xi}^{\omega} \theta_{1-l}\cdot D_{x}^{\omega} r_{-1-k}.   \nonumber
\end{eqnarray}
\noindent

\noindent
Then, $\pi_{j}$ and $q_{j}$ are computed by the following integrals:

\begin{eqnarray}   \label{E:2.34}
\pi_{j} & = & \frac{\partial}{\partial s}\bigg|_{s=0} \int_{{\mathcal N}}  \left( \frac{1}{(2 \pi)^{m-1}} \int_{T_{x}^{\ast}{\mathcal N}}
\frac{1}{2 \pi i} \int_{\gamma} \mu^{-s} \Tr r_{-1-j} \left( x, \xi, \frac{\lambda}{|\lambda|}, \mu\right) d\mu d \xi \right)  d \vol({\mathcal N}) ,        \\
q_{j} & = & - \frac{1}{2} ~ \int_{{\mathcal N}} \left( \frac{1}{(2 \pi)^{m-1}} \int_{T_{x}^{\ast}{\mathcal N}}
\frac{1}{2 \pi i} \int_{\gamma} \mu^{-s} \Tr r_{-1-j} \left(x, \xi, \frac{\lambda}{|\lambda|}, \mu\right) d\mu d \xi \right) d \vol({\mathcal N})  \bigg|_{s=0}.   \nonumber
\end{eqnarray}

\vspace{0.2 cm}
Lemma \ref{Lemma:2.7} shows that the homogeneous symbol of $R_{S}(\lambda)$ is given by

\begin{eqnarray}   \label{E:2.35}
\sigma\big( R_{S}(\lambda) \big) & = & \sigma \big( (Q_{1}(\lambda) + S)^{-1} \big) +
\sigma \big( (Q_{2}(\lambda) - S)^{-1} .
\end{eqnarray}

\noindent
On the boundary normal coordinate system in a collar neighborhood of ${\mathcal N}$ introduced below Definition \ref{Definition:2.2},
the full homogeneous symbols of $Q_{1}(\lambda)$ and $Q_{2}(\lambda)$ are computed in Section 2 of \cite{KL4} in a recursive way.
We denote the symbols of $Q_{1}(\lambda)$ and $Q_{2}(\lambda)$ by

\begin{eqnarray}   \label{E:2.36}
\sigma \big( Q_{1}(\lambda) \big)(x, \xi, \lambda) & \sim & \sum_{j=0}^{\infty} {\frak p}_{1-j}(x, \xi, \lambda) , \qquad
\sigma \big( Q_{2}(\lambda) \big)(x, \xi, \lambda) ~ \sim ~ \sum_{j=0}^{\infty} {\frak q}_{1-j}(x, \xi, \lambda).
\end{eqnarray}

\noindent
It is also well known (cf. \cite{Le1}) that ${\frak p}_{1-j}(x, \xi, \lambda)$ and ${\frak q}_{1-j}(x, \xi, \lambda)$ are
expressed by

\begin{eqnarray}    \label{E:2.37}
{\frak p}_{1-j}(x, \xi, \lambda)  =  A_{1-j}(x, \xi, \lambda) + B_{1-j}(x, \xi, \lambda), \quad
{\frak q}_{1-j}(x, \xi, \lambda)  =  A_{1-j}(x, \xi, \lambda) - B_{1-j}(x, \xi, \lambda),
\end{eqnarray}

\noindent
where $A_{1-j}(x, \xi, \lambda)$ and $B_{1-j}(x, \xi, \lambda)$ are homogeneous of degree $1-j$ with respect to $\xi$ satisfying

\begin{eqnarray}   \label{E:2.38}
A_{1-j}(x, -\xi, \lambda) = (-1)^{j} A_{1-j}(x, \xi, \lambda), \qquad
B_{1-j}(x, -\xi, \lambda) = (-1)^{j+1} B_{1-j}(x, \xi, \lambda).
\end{eqnarray}

\noindent
The homogeneous symbol of $R_{\DN}(\lambda) = Q_{1}(\lambda) + Q_{2}(\lambda) + V_{1}(\lambda) + V_{2}(\lambda)$ (see (\ref{E:2.10})) is given by

\begin{eqnarray}    \label{E:2.39}
\sigma \big( R_{\DN}(\lambda) \big) (x, \xi, \lambda)
& \sim & \sum_{j=0}^{\infty} \left( {\frak p}_{1-j} + {\frak q}_{1-j} \right)(x, \xi, \lambda)
~ = ~ 2 \sum_{j=0}^{\infty} A_{1-j}(x, \xi, \lambda).
\end{eqnarray}

\noindent
The homogeneous symbol of $Q_{1}(\lambda) + S$ and $Q_{2}(\lambda) - S$ are given by

\begin{eqnarray}   \label{E:2.40}
\sigma \big( Q_{1}(\lambda) + S \big) & \sim &
\sum_{j=1}^{\infty} \big( A_{1-j}(x, \xi, \lambda) + {\widetilde B}_{1-j}(x, \xi, \lambda) \big), \\
\sigma \big( Q_{2}(\lambda) - S \big) & \sim &
\sum_{j=1}^{\infty} \big( A_{1-j}(x, \xi, \lambda) - {\widetilde B}_{1-j}(x, \xi, \lambda) \big),    \nonumber
\end{eqnarray}

\noindent
where

\begin{eqnarray}    \label{E:2.41}
{\widetilde B}_{1-j}(x, \xi, \lambda) & = & \begin{cases} B_{1-j}(x, \xi, \lambda) & \text{for} ~~ j \neq 1 \\
B_{0}(x, \xi, \lambda) + S(x)  & \text{for} ~~ j = 1 .   \end{cases}
\end{eqnarray}

\noindent
Then, the homogeneous symbols of $\big( Q_{1}(\lambda) + S \big)^{-1}$ and $\big( Q_{2}(\lambda) - S \big)^{-1}$ are given by

\begin{eqnarray}   \label{E:2.42}
\sigma \big( ( Q_{1}(\lambda) + S )^{-1} \big) & \sim & \sum_{j=0}^{\infty} v_{-1-j} (x, \xi, \lambda),  \qquad
\sigma \big( ( Q_{2}(\lambda) - S )^{-1} \big) ~ \sim ~ \sum_{j=0}^{\infty} w_{-1-j} (x, \xi, \lambda),
\end{eqnarray}

\noindent
where

\begin{eqnarray}    \label{E:2.43}
& & v_{-1} = w_{-1} = \frac{1}{\sqrt{|\xi|^{2} + \lambda}} \Id,  \\
& & v_{-1-j} = - \frac{1}{\sqrt{|\xi|^{2} + \lambda}} \sum_{k=0}^{j-1} \sum_{|\omega|+ \ell + k = j}
\frac{1}{\omega!} \partial_{\xi}^{\omega} \big( A_{1 - \ell} + {\widetilde B}_{1-\ell} \big) \cdot D_{x}^{\omega} v_{-1-k},  \nonumber \\
& & w_{-1-j} = - \frac{1}{\sqrt{|\xi|^{2} + \lambda}} \sum_{k=0}^{j-1} \sum_{|\omega|+ \ell + k = j}
\frac{1}{\omega!} \partial_{\xi}^{\omega} \big( A_{1 - \ell} - {\widetilde B}_{1-\ell} \big) \cdot D_{x}^{\omega} w_{-1-k}.   \nonumber
\end{eqnarray}

\begin{lemma}    \label{Lemma:2.12}
$v_{-1-j} (x, \xi, \lambda)$ and $w_{-1-j} (x, \xi, \lambda)$ are expressed by

\begin{eqnarray*}
v_{-1-j} (x, \xi, \lambda) ~ = ~ {\mathcal V}_{-1-j}(x, \xi, \lambda) + {\mathcal W}_{-1-j}(x, \xi, \lambda), \quad
w_{-1-j} (x, \xi, \lambda) ~ = ~ {\mathcal V}_{-1-j}(x, \xi, \lambda) - {\mathcal W}_{-1-j}(x, \xi, \lambda),
\end{eqnarray*}
where ${\mathcal V}_{-1-j}(x, \xi, \lambda)$ and ${\mathcal W}_{-1-j}(x, \xi, \lambda)$ are homogeneous of degree $-1-j$ with respect to $\xi$ satisfying

\begin{eqnarray*}
{\mathcal V}_{-1-j}(x, -\xi, \lambda) = (-1)^{j} {\mathcal V}_{-1-j}(x, \xi, \lambda), \qquad
{\mathcal W}_{1-j}(x, -\xi, \lambda) = (-1)^{j+1} {\mathcal W}_{-1-j}(x, \xi, \lambda).
\end{eqnarray*}

\noindent
Hence, the homogeneous symbol of $R_{S}(\lambda)$ is given by

\begin{eqnarray*}
\sigma \big( R_{S}(\lambda) \big) (x, \xi, \lambda)
& \sim & 2 ~ \sum_{j=0}^{\infty} {\mathcal V}_{-1-j}(x, \xi, \lambda).
\end{eqnarray*}
\end{lemma}

\begin{proof}
We use the induction. Setting

\begin{eqnarray}    \label{E:2.44}
{\mathcal V}_{-1} = \frac{1}{\sqrt{|\xi|^{2} + \lambda}} \Id, \qquad  {\mathcal W}_{-1} = 0,
\end{eqnarray}

\noindent
the statement holds for $j = 0$. We assume that the statement holds for $0 \leq j \leq j_{0}$. Then,

\begin{eqnarray}     \label{E:2.45}
v_{-1-(j_{0}+1)} & = & -\frac{1}{\sqrt{|\xi|^{2} + \lambda}} \sum_{k=0}^{j_{0}} \sum_{|\omega|+ \ell + k = j_{0} + 1}
\frac{1}{\omega!} \partial_{\xi}^{\omega} \big( A_{1 - \ell} + {\widetilde B}_{1 - \ell} \big) \cdot D_{x}^{\omega}
\big( {\mathcal V}_{-1-k} + {\mathcal W}_{-1-k} \big)  \\
& = & -\frac{1}{\sqrt{|\xi|^{2} + \lambda}} \sum_{k=0}^{j_{0}} \sum_{|\omega|+ \ell + k = j_{0} + 1}  \bigg\{
\frac{1}{\omega!} \bigg( \partial_{\xi}^{\omega} A_{1 - \ell} \cdot D_{x}^{\omega} {\mathcal V}_{-1-k} +
\partial_{\xi}^{\omega}{\widetilde B}_{1 - \ell} \cdot D_{x}^{\omega} {\mathcal W}_{-1-k} \bigg)    \nonumber  \\
& & + \frac{1}{\omega!} \bigg( \partial_{\xi}^{\omega} A_{1 - \ell} \cdot D_{x}^{\omega} {\mathcal W}_{-1-k} +
\partial_{\xi}^{\omega} {\widetilde B}_{1 - \ell} \cdot D_{x}^{\omega} {\mathcal V}_{-1-k} \bigg)  \bigg\},     \nonumber
\end{eqnarray}
\begin{eqnarray*}
w_{-1-(j_{0}+1) } & = & -\frac{1}{\sqrt{|\xi|^{2} + \lambda}} \sum_{k=0}^{j_{0}} \sum_{|\omega|+ \ell + k = j_{0} + 1}  \bigg\{
\frac{1}{\omega!} \bigg( \partial_{\xi}^{\omega} A_{1 - \ell} \cdot D_{x}^{\omega} {\mathcal V}_{-1-k} +
\partial_{\xi}^{\omega}{\widetilde B}_{1 - \ell} \cdot D_{x}^{\omega} {\mathcal W}_{-1-k} \bigg)     \nonumber \\
& & - \frac{1}{\omega!} \bigg( \partial_{\xi}^{\omega} A_{1 - \ell} \cdot D_{x}^{\omega} {\mathcal W}_{-1-k} +
\partial_{\xi}^{\omega} {\widetilde B}_{1-\ell} \cdot D_{x}^{\omega} {\mathcal V}_{-1-k} \bigg)  \bigg\}.   \nonumber
\end{eqnarray*}

\noindent
Setting

\begin{eqnarray}    \label{E:2.46}
{\mathcal V}_{-1-(j_{0}+1)}  =  -\frac{1}{\sqrt{|\xi|^{2} + \lambda}} \sum_{k=0}^{j_{0}} \sum_{|\omega|+ \ell + k = j_{0} + 1}
\frac{1}{\omega!} \bigg( \partial_{\xi}^{\omega} A_{1 - \ell} \cdot D_{x}^{\omega} {\mathcal V}_{-1-k} +
\partial_{\xi}^{\omega}{\widetilde B}_{1 - \ell} \cdot D_{x}^{\omega} {\mathcal W}_{-1-k} \bigg),   \\
{\mathcal W}_{-1-(j_{0}+1)}  = - \frac{1}{\sqrt{|\xi|^{2} + \lambda}} \sum_{k=0}^{j_{0}} \sum_{|\omega|+ \ell + k = j_{0} + 1}
\frac{1}{\omega!} \bigg( \partial_{\xi}^{\omega} A_{1 - \ell} \cdot D_{x}^{\omega} {\mathcal W}_{-1-k} +
\partial_{\xi}^{\omega} {\widetilde B}_{1 - \ell} \cdot D_{x}^{\omega} {\mathcal V}_{-1-k} \bigg),    \nonumber
\end{eqnarray}

\noindent
this completes the proof.
\end{proof}

\begin{corollary}   \label{Corollary:2.13}
The homogeneous symbols of $R_{S}(\lambda)^{-1}$ and $\big(\mu - R_{S}(\lambda)^{-1} \big)^{-1}$ satisfy the following property.

\begin{eqnarray*}
\theta_{1-j}(x, -\xi, \lambda, \mu) = (-1)^{j} \theta_{1-j}(x, \xi, \lambda, \mu), \qquad
r_{-1-j}(x, -\xi, \lambda, \mu) = (-1)^{j} r_{-1-j}(x, \xi, \lambda, \mu).
\end{eqnarray*}
\end{corollary}

\noindent
Corollary \ref{Corollary:2.14} shows that if $j$ is odd, then $r_{-1-j}(x, \xi, \lambda, \mu) $
is an odd function with respect to $\xi$. Hence, if $j$ is odd, then $\pi_{j} = q_{j} = 0$,
which yields the following result.

\begin{corollary}  \label{Corollary:2.14}
The constant $- a_{0}$ in Theorem \ref{Theorem:2.11} is the zero coefficient in the asymptotic expansion of $\ln \Det R_{S}(\lambda)$ for $|\lambda| \rightarrow \infty$.
If $\Dim M$ is even, then $a_{0} = 0$.
\end{corollary}

\noindent
{\it Remark} : The argument presented in \cite{KL3} shows that $P(\lambda) = 0$ when $\Dim M$ is even.
Moreover, $\pi_{j}$, $q_{j}$ and hence coefficients of $P(\lambda)$ are expressed by integration of some local densities consisting of curvature tensors (cf. \cite{KL4}, \cite{PS}), which will be discussed elsewhere.
\newline

\vspace{0.2 cm}
If $\Delta_{M, B}$ and $\Delta_{M_{0}, {\frak R}_{0}, B}$ are invertible operators, then $R_{S}(0)$ is also invertible by Lemma \ref{Lemma:2.8} and hence
we may simply
substitute $\lambda$ for $0$ in Theorem \ref{Theorem:2.11}, which leads to the following result.

\begin{corollary}    \label{Corollary:2.15}
Suppose that $\Delta_{M, B}$ and $\Delta_{M_{0}, {\frak R}_{0}, B}$ are invertible operators. Then,
\begin{eqnarray*}
\ln \Det \Delta_{M, B} - \ln \Det \Delta_{M_{0}, {\frak R}_{0}, B} & = & a_{0} + \ln \Det R_{S}(0),
\end{eqnarray*}
where $- a_{0}$ is the constant term in the asymptotic expansion of $\ln \Det R_{S}(\lambda)$ for $|\lambda| \rightarrow \infty$.
\end{corollary}

\vspace{0.2 cm}
\noindent
In many cases, $\Delta_{M, B}$ may not be invertible. For example, if $M$ is a closed manifold, $\Delta_{M}$ may have finite dimensional nontrivial kernel.
In the next section we are going to discuss the asymptotic behaviors of Theorem \ref{Theorem:2.11} for $|\lambda| \rightarrow 0$ when $\Delta_{M, B}$ is not invertible in a simpler case that $M_{0}$ has at least two components.

\vspace{0.2 cm}
In the remaining part of this section
we are going to compute the constant $a_{0}$ in Theorem \ref{Theorem:2.11} when $S = \alpha \Id$ for $\alpha \in {\mathbb R}$ and $M$ has a product structure near ${\mathcal N}$ so that $\Delta_{M}$ is
$- \partial_{x_{m}}^{2} + \Delta_{{\mathcal N}}$ on a collar neighborhood of ${\mathcal N}$, where $\Delta_{{\mathcal N}}$ is a Laplacian on ${\mathcal N}$. For simplicity, we assume that ${\mathcal N}$ is connected.
In this case, it is known (\cite{Le2}, \cite{PW}) that $Q_{1}(\lambda)$ and $Q_{2}(\lambda)$ differ from $\sqrt{\Delta_{{\mathcal N}} + \lambda}$ by smoothing operators,
so that $R_{\DN}(\lambda)$ differs from $2 \sqrt{\Delta_{{\mathcal N}} + \lambda}$ by a smoothing operator.
Hence, by Lemma \ref{Lemma:2.7} we get

\begin{eqnarray}   \label{E:2.47}
R_{S}(\lambda)^{-1} & = &  \left( Q_{1}(\lambda) + \alpha \right) R_{\DN}(\lambda)^{-1} \left( Q_{2}(\lambda) - \alpha \right) + {\mathcal K}_{1}(\lambda)  \\
& = & \frac{1}{2} \left( \sqrt{\Delta_{{\mathcal N}} + \lambda} - \alpha^{2} \sqrt{\Delta_{{\mathcal N}} + \lambda}^{-1} \right) ~ + ~ {\mathcal K}_{2}(\lambda) ~ =: ~ {\mathcal D}(\lambda) + {\mathcal K}_{2}(\lambda),   \nonumber
\end{eqnarray}

\noindent
where ${\mathcal D}(\lambda) = \frac{1}{2} \left( \sqrt{\Delta_{{\mathcal N}} + \lambda} - \alpha^{2} \sqrt{\Delta_{{\mathcal N}} + \lambda}^{-1} \right)$ and ${\mathcal K}_{i}(\lambda)$'s are smoothing operators.
It is shown in the Appendix of \cite{BFK1} that the asymptotic expansion of $\ln \Det R_{S}(\lambda)^{-1}$ for $|\lambda| \rightarrow \infty$ is determined up to smoothing operators, and hence it is enough to consider
$\ln \Det {\mathcal D}(\lambda)$.
We here note that $\Delta_{{\mathcal N}} + \lambda$ is an elliptic operator with parameter $\lambda$ of weight $2$ so that
the leading symbol $\sigma_{L}(\Delta_{{\mathcal N}} + \lambda)$ is $\sigma_{L}(\Delta_{{\mathcal N}}) + \lambda$.
Simple computation shows that

\begin{eqnarray}   \label{E:2.48}
 \ln \Det {\mathcal D}(\lambda)
~ = ~  - \ln 2 ~ \zeta_{\left( \sqrt{\Delta_{{\mathcal N}} + \lambda} - \alpha^{2} \sqrt{\Delta_{{\mathcal N}} + \lambda}^{-1} \right)}(0)  +  \ln \Det \left( \sqrt{\Delta_{{\mathcal N}} + \lambda} - \alpha^{2} \sqrt{\Delta_{{\mathcal N}} + \lambda}^{-1} \right).
\end{eqnarray}

To compute $\zeta_{\left( \sqrt{\Delta_{{\mathcal N}} + \lambda} - \alpha^{2} \sqrt{\Delta_{{\mathcal N}} + \lambda}^{-1} \right)}(0)$ and
$\ln \Det \left( \sqrt{\Delta_{{\mathcal N}} + \lambda} - \alpha^{2} \sqrt{\Delta_{{\mathcal N}} + \lambda}^{-1} \right)$, we are going to use
the arguments presented in \cite{EVZ}.
Let $\{\mu_{j} \mid j \in {\mathbb N} \}$ be the spectrum of $\Delta_{{\mathcal N}}$.
Using the binomial series, it follows for $|\lambda| \gg 0$ and $\Re s \gg 0$ that

\begin{eqnarray}   \label{E:2.49}
& & \zeta_{\left( \sqrt{\Delta_{{\mathcal N}} + \lambda} - \alpha^{2} \sqrt{\Delta_{{\mathcal N}} + \lambda}^{-1} \right)}(s) \\
& = &
\sum_{j=1}^{\infty} \left( \sqrt{\mu_{j} + \lambda} - \frac{\alpha^{2}}{\sqrt{\mu_{j} + \lambda}} \right)^{-s} ~ = ~
\sum_{j=1}^{\infty} (\mu_{j} + \lambda)^{- \frac{s}{2}} \left( 1 - \frac{\alpha^{2}}{\mu_{j} + \lambda} \right)^{-s}   \nonumber  \\
& = & \sum_{j=1}^{\infty} (\mu_{j} + \lambda)^{- \frac{s}{2}}
\sum_{k=0}^{\infty} \frac{\Gamma(s+k) \alpha^{2k}}{k! \Gamma(s)} (\mu_{j} + \lambda)^{-k}
~ = ~ \zeta_{(\Delta_{{\mathcal N}} + \lambda)} \left( \frac{s}{2} \right) ~ + ~
\sum_{k=1}^{\infty} \frac{\Gamma(s+k) \alpha^{2k}}{k! \Gamma(s)} \zeta_{(\Delta_{{\mathcal N}} + \lambda)} \left(\frac{s}{2} + k \right)
\nonumber \\
& = & \zeta_{(\Delta_{{\mathcal N}} + \lambda)} \left( \frac{s}{2} \right) ~ + ~
\frac{2}{\Gamma(s+1)} \sum_{k=1}^{\infty} \frac{\Gamma(s+k) \alpha^{2k}}{k!} \cdot \frac{s}{2} \cdot \zeta_{(\Delta_{{\mathcal N}} + \lambda)} \left(\frac{s}{2} + k \right),   \nonumber
\end{eqnarray}

\noindent
which shows that

\begin{eqnarray}     \label{E:2.50}
\zeta_{\left( \sqrt{\Delta_{{\mathcal N}} + \lambda} - \alpha^{2} \sqrt{\Delta_{{\mathcal N}} + \lambda}^{-1} \right)}(0) & = &
\zeta_{(\Delta_{{\mathcal N}} + \lambda)} \left( 0 \right) ~ + ~
2 \sum_{k=1}^{[\frac{m-1}{2}]} \frac{\alpha^{2k}}{k} \cdot \Res_{s=k} \zeta_{(\Delta_{{\mathcal N}} + \lambda)} \left( s \right),
\end{eqnarray}

\noindent
which is a polynomial of $\lambda$.
We denote the small time asymptotics of the heat trace of $\Delta_{{\mathcal N}}$ by

\begin{eqnarray}   \label{E:2.51}
\Tr e^{-t \Delta_{{\mathcal N}}} & \sim & \sum_{j=0}^{\infty} {\frak a}_{j} t^{- \frac{m-1}{2} + j}.
\end{eqnarray}

\noindent
Then, the constant term ${\frak w}_{0}$ in the asymptotic expansion of
$\zeta_{\left( \sqrt{\Delta_{{\mathcal N}} + \lambda} - \alpha^{2} \sqrt{\Delta_{{\mathcal N}} + \lambda}^{-1} \right)}(0)$ for $|\lambda| \rightarrow \infty$ is

\begin{eqnarray}    \label{E:2.52}
{\frak w}_{0} & = &
\begin{cases} 0 & \text{for} \quad {m-1} \quad \text{odd}  \\
{\frak a}_{\frac{m-1}{2}} + 2 \sum_{k=1}^{\frac{m-1}{2}} \frac{1}{k !} \cdot {\frak a}_{\frac{m-1}{2} - k} \cdot \alpha^{2k}
& \text{for} \quad {m-1} \quad \text{even}.  \end{cases}
\end{eqnarray}

\noindent
Using the same method, it follows that

\begin{eqnarray}   \label{E:2.53}
\zeta_{\left( \sqrt{\Delta_{{\mathcal N}} + \lambda} + \alpha  \right)}(s) & = & \sum_{j=1}^{\infty} ( \sqrt{\mu_{j} + \lambda} + \alpha)^{-s}
~ = ~ \sum_{j=1}^{\infty} (\mu_{j} + \lambda)^{-\frac{s}{2}} \left( 1 + \frac{\alpha}{\sqrt{\mu_{j} + \lambda}} \right)^{-s}  \\
& = & \sum_{j=1}^{\infty} (\mu_{j} + \lambda)^{-\frac{s}{2}} \sum_{k=0}^{\infty} \frac{(-1)^{k} \Gamma(s+k)}{k! \Gamma(s)}
\left( \frac{\alpha}{\sqrt{\mu_{j} + \lambda}} \right)^{k}      \nonumber \\
& = & \zeta_{(\Delta_{{\mathcal N}} + \lambda)}\left(\frac{s}{2} \right) ~ + ~
\frac{1}{\Gamma(s)} \sum_{k=1}^{\infty} \frac{(-1)^{k} \alpha^{k} \Gamma(s+k)}{k!} \zeta_{(\Delta_{{\mathcal N}} + \lambda)}\left(\frac{s}{2} + \frac{k}{2} \right),      \nonumber \\
\zeta_{\left( \sqrt{\Delta_{{\mathcal N}} + \lambda} - \alpha  \right)}(s)
& = & \zeta_{(\Delta_{{\mathcal N}} + \lambda)}\left(\frac{s}{2} \right) ~ + ~
\frac{1}{\Gamma(s)} \sum_{k=1}^{\infty} \frac{\alpha^{k} \Gamma(s+k)}{k!} \zeta_{(\Delta_{{\mathcal N}} + \lambda)}\left(\frac{s}{2} + \frac{k}{2} \right),   \nonumber
\end{eqnarray}

\noindent
which shows that

\begin{eqnarray}    \label{E:2.54}
& & \zeta_{\left( \sqrt{\Delta_{{\mathcal N}} + \lambda} + \alpha  \right)}(s) +
\zeta_{\left( \sqrt{\Delta_{{\mathcal N}} + \lambda} - \alpha  \right)}(s)  \\
& = &  2 \zeta_{(\Delta_{{\mathcal N}} + \lambda)}\left(\frac{s}{2} \right) ~ + ~
\frac{2}{\Gamma(s)} \sum_{k=1}^{\infty} \frac{\alpha^{2k} \Gamma(s+2k)}{(2k)!} \zeta_{(\Delta_{{\mathcal N}} + \lambda)}\left(\frac{s}{2} + k \right) \nonumber \\
& = & 2 \zeta_{(\Delta_{{\mathcal N}} + \lambda)}\left(\frac{s}{2} \right) ~ + ~
\frac{2}{\Gamma(s + 1)} \sum_{k=1}^{\infty} \frac{\alpha^{2k}}{k} \frac{\Gamma(s+2k)}{\Gamma(2k)} \cdot \frac{s}{2} \cdot \zeta_{(\Delta_{{\mathcal N}} + \lambda)}\left(\frac{s}{2} + k \right).     \nonumber
\end{eqnarray}

\noindent
By (\ref{E:2.49}) and (\ref{E:2.54}), it follows that

\begin{eqnarray}   \label{E:2.55}
& & \zeta_{\left( \sqrt{\Delta_{{\mathcal N}} + \lambda} - \alpha^{2} \sqrt{\Delta_{{\mathcal N}} + \lambda}^{-1} \right)}(s) -
\left\{ \zeta_{\left( \sqrt{\Delta_{{\mathcal N}} + \lambda} + \alpha  \right)}(s) +
\zeta_{\left( \sqrt{\Delta_{{\mathcal N}} + \lambda} - \alpha  \right)}(s) -
\zeta_{(\Delta_{{\mathcal N}} + \lambda)}\left(\frac{s}{2} \right) \right\}  \\
& = & \frac{2}{\Gamma(s+1)} \sum_{k=1}^{\infty} \frac{\alpha^{2k}}{k}
\left\{ \frac{\Gamma(s+k)}{\Gamma(k)} - \frac{\Gamma(s+2k)}{\Gamma(2k)} \right\} \cdot \frac{s}{2} \cdot \zeta_{(\Delta_{{\mathcal N}} + \lambda)} \left(\frac{s}{2} + k \right).   \nonumber
\end{eqnarray}

\noindent
Hence, we get

\begin{eqnarray}    \label{E:2.56}
& & \ln \Det \left( \sqrt{\Delta_{{\mathcal N}} + \lambda} - \alpha^{2} \sqrt{\Delta_{{\mathcal N}} + \lambda}^{-1} \right)
~ = ~  \ln \Det (\sqrt{\Delta_{{\mathcal N}} + \lambda} + \alpha )   \\
 & & + \ln \Det (\sqrt{\Delta_{{\mathcal N}} + \lambda} - \alpha )
 - \frac{1}{2} \ln \Det (\Delta_{{\mathcal N}} + \lambda) ~ + ~ 2 \sum_{k=1}^{[\frac{m-1}{2}]} \frac{\alpha^{2k}}{k}
\left( \sum_{p=1}^{2k-1} \frac{1}{p} - \sum_{p=1}^{k-1} \frac{1}{p} \right) \cdot
\Res_{s=k} \zeta_{(\Delta_{{\mathcal N}} + \lambda)}(s),   \nonumber
\end{eqnarray}

\noindent
where the constant term in the asymptotic expansions of $\frac{1}{2} \ln \Det (\Delta_{{\mathcal N}} + \lambda)$
is $0$ and that of $\Res_{s=k} \zeta_{(\Delta_{{\mathcal N}} + \lambda)}(s)$ is  $\frac{1}{(k-1)!} {\frak a}_{\frac{m-1}{2} - k}$ for $m$ odd and $0$ for $m$ even.

We next compute the constant terms in the asymptotic expansions of $\ln \Det (\sqrt{\Delta_{{\mathcal N}} + \lambda} ~ \pm  \alpha )$ for
$|\lambda| \rightarrow \infty$.
For two functions $f(s, \lambda)$ and $g(s, \lambda)$, we define an equivalence relation "$\approx$" as follows.

\begin{eqnarray*}
f(s, \lambda) ~ \approx ~ g(x, \lambda) \quad \text{if and only if} \quad
\lim_{|\lambda| \rightarrow \infty} \frac{\partial}{\partial s}|_{s=0} \bigg( f(s, \lambda) - g(s, \lambda) \bigg) = 0.
\end{eqnarray*}

\noindent
We note that

\begin{eqnarray}     \label{E:2.57}
& & \zeta_{\left( \sqrt{\Delta_{{\mathcal N}} + \lambda} ~ + ~ \alpha \right)}(s) ~ = ~
\frac{1}{\Gamma(s)} \int_{0}^{\infty} t^{s-1} \Tr e^{- t \left( \sqrt{\Delta_{{\mathcal N}} + \lambda} ~ + ~ \alpha \right)} ~ dt \\
& \approx & \frac{1}{\Gamma(s)} \int_{0}^{1} t^{s-1} e^{- \alpha t} \Tr e^{- t \sqrt{\Delta_{{\mathcal N}} + \lambda}} ~ dt
~ \approx ~ \sum_{k=0}^{m-1} \frac{(- \alpha)^{k}}{k !} \frac{1}{\Gamma(s)} \int_{0}^{1} t^{s+k-1} \Tr e^{- t \sqrt{\Delta_{{\mathcal N}} + \lambda}} ~ dt
  \nonumber \\
& \approx & \sum_{k=0}^{m-1} \frac{(- \alpha)^{k}}{k !} \frac{\Gamma(s+k)}{\Gamma(s)} \frac{1}{\Gamma(\frac{s+k}{2})}
\int_{0}^{1} t^{\frac{s+k}{2}-1} \Tr e^{-t(\Delta_{{\mathcal N}} + \lambda)} ~ dt    \nonumber \\
& \approx & \sum_{k=0}^{m-1} \frac{(- \alpha)^{k}}{k !} \frac{\Gamma(s+k)}{\Gamma(s)} \frac{1}{\Gamma(\frac{s+k}{2})}
\int_{0}^{1} t^{\frac{s+k}{2}-1} e^{- t \lambda} \sum_{j=0}^{[\frac{m-1}{2}]} {\frak a}_{j} t^{-\frac{m-1}{2} + j} ~ dt  \nonumber \\
& \approx & \sum_{k=0}^{m-1} \sum_{j=0}^{[\frac{m-1}{2}]} {\frak a}_{j} \frac{(- \alpha)^{k}}{k !} \frac{\Gamma(s+k)}{\Gamma(s)} \frac{1}{\Gamma(\frac{s+k}{2})}
\int_{0}^{\infty} t^{\frac{s+k + 2j - m+1}{2}-1} e^{- t \lambda} ~ dt     \nonumber \\
& = & \sum_{k=0}^{m-1} \sum_{j=0}^{[\frac{m-1}{2}]} {\frak a}_{j} \frac{(- \alpha)^{k}}{k !} \frac{\Gamma(s+k)}{\Gamma(s)}
\frac{\Gamma(\frac{s+k + 2j - m+1}{2})}{\Gamma(\frac{s+k}{2})} \lambda^{-\frac{s+k + 2j - m+1}{2}}.  \nonumber
\end{eqnarray}

\noindent
Hence, the constant term ${\frak s}_{\alpha}$ in the asymptotic expansion of $\ln \Det \big( \sqrt{\Delta_{{\mathcal N}} + \lambda} + \alpha \big)$ is given by

\begin{eqnarray}   \label{E:2.58}
{\frak s}_{\alpha} & = & - \sum_{\stackrel{k + 2j = m-1}{k \geq 1, j \geq 0}}
{\frak a}_{j} \cdot \frac{(-1)^{k} \alpha^{k}}{k!} \frac{d}{ds}\big|_{s=0}
\left( \frac{\Gamma(s+k)}{\Gamma(s)} \frac{\Gamma(\frac{s}{2})}{\Gamma(\frac{s+k}{2})} \right)    \\
& = & - \sum_{\stackrel{k + 2j = m-1}{k \geq 1, j \geq 0}} \frac{(-1)^{k} {\frak a}_{j} \cdot \alpha^{k}}{k} \cdot \frac{1}{\Gamma(\frac{k}{2})}
\left( 2 \sum_{p=1}^{k-1} \frac{1}{p} - {\frak c}_{k} \right),  \nonumber
\end{eqnarray}

\noindent
where

\begin{eqnarray}    \label{E:2.59}
{\frak c}_{k} & = & \begin{cases} \sum_{p=1}^{\frac{k}{2}-1} \frac{1}{p} & \text{for} \quad k \quad \text{even}  \\
-2 \ln 2 + 2 \sum_{p=1}^{[\frac{k}{2}]} \frac{1}{2p - 1}  & \text{for} \quad k \quad \text{odd} . \end{cases}
\end{eqnarray}

\noindent
Here $\sum_{p=1}^{\frac{k}{2}-1} \frac{1}{p}$ or $\sum_{p=1}^{[\frac{k}{2}]} \frac{1}{2p - 1}$ is understood to be zero when $\frac{k}{2} = 1$ or $[\frac{k}{2}]=0$.
Similarly,

\begin{eqnarray}   \label{E:2.60}
{\frak s}_{- \alpha} & = & - \sum_{\stackrel{k + 2j = m-1}{k \geq 1, j \geq 0}}
\frac{{\frak a}_{j} \cdot \alpha^{k}}{k} \cdot \frac{1}{\Gamma(\frac{k}{2})}
\left( 2 \sum_{p=1}^{k-1} \frac{1}{p} - {\frak c}_{k} \right),
\end{eqnarray}

\noindent
which shows that

\begin{eqnarray}   \label{E:2.61}
{\frak s}_{\alpha} + {\frak s}_{- \alpha}  =  \begin{cases} 0 & \text{for} \quad m-1 \quad \text{odd} \\
- \sum_{k=1}^{\frac{m-1}{2}} \frac{1}{k !} \cdot {\frak a}_{(\frac{m-1}{2} - k)} \cdot \alpha^{2k} \cdot
\left( 2 \sum_{p=1}^{2k-1} \frac{1}{p} - \sum_{p=1}^{k-1} \frac{1}{p} \right) &  \text{for} \quad  m-1 \quad \text{even}.  \end{cases}
\end{eqnarray}

\noindent
By (\ref{E:2.56}) and (\ref{E:2.61}), the constant term ${\frak w}_{1}$ for
$\ln \Det \left( \sqrt{\Delta_{{\mathcal N}} + \lambda} - \alpha^{2} \sqrt{\Delta_{{\mathcal N}} + \lambda}^{-1} \right)$ for $|\lambda| \rightarrow \infty$ is

\begin{eqnarray}   \label{E:2.62}
{\frak w}_{1} & = & \begin{cases} 0 & \text{for} \quad m-1 \quad \text{odd}  \\
- \sum_{k=2}^{\frac{m-1}{2}} \frac{1}{k!} \cdot {\frak a}_{(\frac{m-1}{2} - k)} \cdot \alpha^{2k} \cdot \sum_{p=1}^{k-1} \frac{1}{p}
& \text{for} \quad m-1 \quad \text{even}.  \end{cases}
\end{eqnarray}

\noindent
If ${\mathcal N}$ consists of $k_{0}$ components, {\it i.e.} ${\mathcal N} = \cup_{i=1}^{k_{0}} {\mathcal N}_{i}$, then
$\sqrt{\Delta_{{\mathcal N}} + \lambda}$ is given by

\begin{eqnarray}     \label{E:2.77}
\sqrt{\Delta_{{\mathcal N}} + \lambda} & = & \left( \begin{array}{clcr} \sqrt{\Delta_{{\mathcal N}_{1}} + \lambda} & 0 & 0 \\ 0 & \ddots & 0 \\
0 & 0 & \sqrt{\Delta_{{\mathcal N}_{k_{0}}} + \lambda} \end{array} \right) : \oplus_{i=1}^{k_{0}} C^{\infty}({\mathcal E}|_{{\mathcal N}_{i}}) \rightarrow \oplus_{i=1}^{k_{0}} C^{\infty}({\mathcal E}|_{{\mathcal N}_{i}}).
\end{eqnarray}

\noindent
For $\alpha_{i} \in {\mathbb R}$, $1 \leq i \leq k_{0}$, we choose

\begin{eqnarray}     \label{E:2.78}
S & = & \left( \begin{array}{clcr} \alpha_{1} & 0 & 0 \\ 0 & \ddots & 0 \\ 0 & 0 & \alpha_{k_{0}} \end{array} \right).
\end{eqnarray}

\noindent
The following result is obtained by (\ref{E:2.48}), (\ref{E:2.52}) and (\ref{E:2.62}).

\begin{theorem}   \label{Theorem:2.16}
Suppose that ${\mathcal N} = \cup_{i=1}^{k_{0}} {\mathcal N}_{i}$ has $k_{0}$ components and $M$ has a product structure near ${\mathcal N}$ so that
$\sqrt{\Delta_{N} + \lambda}$ is given by
(\ref{E:2.77}) and $S$ is defined by (\ref{E:2.78}). Then, the constant $a_{0}$ in Theorem \ref{Theorem:2.11} is given by
\begin{eqnarray*}
a_{0} ~ = ~ \sum_{i=1}^{k_{0}} a_{0}^{i}, \quad \text{where} \quad
a^{i}_{0} ~ = ~ \begin{cases} 0 & \text{for} \quad m-1 \quad \text{odd}  \\
- \ln 2 \cdot
\left\{ {\frak a}^{i}_{\frac{m-1}{2}} + 2 \sum_{k=1}^{\frac{m-1}{2}} \frac{1}{k !} \cdot {\frak a}^{i}_{\left(\frac{m-1}{2} - k \right)} \cdot \alpha_{i}^{2k} \right\}  \\
\hspace{1.0 cm}
- \sum_{k=2}^{\frac{m-1}{2}} \frac{1}{k!} \cdot {\frak a}^{i}_{(\frac{m-1}{2} - k)} \cdot \alpha_{i}^{2k} \cdot \sum_{p=1}^{k-1} \frac{1}{p}
 & \text{for} \quad m-1 \quad \text{even}.
\end{cases}
\end{eqnarray*}
\end{theorem}

\vspace{0.2 cm}

\noindent
{\it Remark} : If $\alpha_{i} = 0$ for all $i$, ${\frak R}_{1}$ and ${\frak R}_{2}$ give the Neumann boundary condition.
The above theorem shows that $a_{0}$ is the same as in the case of the Dirichlet boundary condition,
which is zero for $m-1$ odd and $a_{0}  =  - \sum_{i=1}^{k_{0}} \ln 2 \cdot \left( \zeta_{\Delta_{{\mathcal N}_{i}}}(0) + \Dim \Ker \Delta_{{\mathcal N}_{i}} \right)$
for $m-1$ even (\cite{KL1}, \cite{Le2}, \cite{PW}).

\vspace{0.3 cm}

\section{The gluing formula when $M_{0}$ is not connected}

In this section we assume that $M_{0}$ has at least two components $M_{1}$ and $M_{2}$ so that $M_{0} = M_{1} \cup M_{2}$. In this case, $R_{S}(\lambda)$ is given by Lemma \ref{Lemma:2.7} and $\ln \Det ( \Delta_{M_{0}, {\frak R}_{0}, B} + \lambda)$ is described by

\begin{eqnarray}     \label{3.1}
\ln \Det ( \Delta_{M_{0}, {\frak R}_{0}, B} + \lambda) ~ = ~ \ln \Det ( \Delta_{M_{1}, {\frak R}_{1}, B} + \lambda) + \ln \Det ( \Delta_{M_{2}, {\frak R}_{2}, B} + \lambda).
\end{eqnarray}

\noindent
We let $\Dim \Ker \Delta_{M, B} = \ell_{0}$ and assume that both $\Delta_{M_{1}, {\frak R}_{1}, B}$
and $\Delta_{M_{2}, {\frak R}_{2}, B}$ are invertible operators.
For $\lambda \rightarrow 0$ and $i = 1, 2$, we have

\begin{eqnarray}   \label{E:3.2}
& & \ln \Det ( \Delta_{M, B} + \lambda) ~ = ~ \ln \lambda^{\ell_{0}} + \ln \Det^{\ast} \Delta_{M, B} + O(\lambda),    \\
& & \ln \Det ( \Delta_{M_{i}, {\frak R}_{i}, B} + \lambda) ~ = ~ \ln \Det \Delta_{M_{i}, {\frak R}_{i}, B} + O(\lambda).   \nonumber
\end{eqnarray}

\noindent
It follows from (\ref{E:2.7}) that

\begin{eqnarray}     \label{E:3.3}
\Ker ( Q_{1}(0) + S ) ~ = ~ \{ \phi|_{{\mathcal N}_{1}} \mid \phi \in \Ker \Delta_{M_{1}, {\frak R}_{1}, B} \}, \quad
\Ker ( Q_{2}(0) - S ) ~ = ~ \{ \psi|_{{\mathcal N}_{2}} \mid \psi \in \Ker \Delta_{M_{2}, {\frak R}_{2}, B} \},
\end{eqnarray}

\noindent
which shows that
$\Delta_{M_{1}, {\frak R}_{1}, B}$ and $\Delta_{M_{2}, {\frak R}_{2}, B}$ are invertible if and only if
$Q_{1}(0) + S$ and $Q_{2}(0) - S$ are invertible.
In this section we assume that both $Q_{1}(0) + S$ and $Q_{2}(0) - S$ are invertible operators.

From the invertibility of $ Q_{1}(0) + S$ and $ Q_{2}(0) - S$, there exists $\epsilon > 0$ such that both
$ Q_{1}(\lambda) + S$ and $ Q_{2}(\lambda) - S$ are invertible for $|\lambda| < \epsilon$ so that $R_{S}(\lambda)$ is well defined for $|\lambda| < \epsilon$.
Moreover, $R_{S}(\lambda)$ is self-adjoint  for $\lambda \in {\mathbb R}$, which implies that
$R_{S}(\lambda)$ is a self-adjoint holomorphic family of type (A) in the sense of T. Kato, whose definition one can find in p.375 and
p.386 of \cite{Ka}.
Similarly, $R_{\DN}(\lambda)$ is also a self-adjoint holomorphic family of type (A).
Theorem 3.9 in p.392 of \cite{Ka} tells that the eigenvalues and corresponding eigensections of $R_{S}(\lambda)$ and $R_{\DN}(\lambda)$ are holomorphic.
More precisely, there exist holomorphic families $\{ \kappa_{j}(\lambda) \mid j = 1, 2, \cdots \}$ and $\{ \varphi_{j}(\lambda) \mid j = 1, 2, \cdots \}$ of
eigenvalues and eigensections of $R_{S}(\lambda)$ such that

\begin{eqnarray*}
\parallel \varphi_{j}(\lambda) \parallel = 1, \qquad \lim_{\lambda \rightarrow 0} \kappa_{j}(\lambda) = 0 \quad \text{for} \quad 1 \leq j \leq \ell_{0},
\end{eqnarray*}

\noindent
 and
 $\{ \varphi_{1}(0), \cdots, \varphi_{\ell_{0}}(0) \}$ is an orthonormal basis of $\Ker R_{S}(0)$ by Lemma \ref{Lemma:2.8}.
We get

\begin{eqnarray}   \label{E:3.4}
\ln \Det R_{S}(\lambda) & = & \ln \kappa_{1}(\lambda) \cdots \kappa_{\ell_{0}}(\lambda) +
\ln \Det^{\ast} R_{S}(0) + O(\lambda) \qquad (\Mod ~ 2\pi i).
\end{eqnarray}

\noindent
Similarly, there exist holomorphic families $\{ \tau_{j}(\lambda) \mid j = 1, 2, \cdots \}$ and $\{ \omega_{j}(\lambda) \mid j = 1, 2, \cdots \}$ of eigenvalues and eigensections of $R_{\DN}(\lambda)$ such that

\begin{eqnarray*}
\parallel \omega_{j}(\lambda)  \parallel = 1, \qquad \lim_{\lambda \rightarrow 0} \tau_{j}(\lambda) = 0 \quad \text{for} \quad 1 \leq j \leq \ell_{0},
\end{eqnarray*}

\noindent
 and
 $\{ \omega_{1}(0), \cdots, \omega_{\ell_{0}}(0) \}$ is an orthonormal basis of $\Ker R_{\DN}(0)$. Then,

\begin{eqnarray*}
\{ (Q_{1}(0) + S) \omega_{1}(0), \cdots, (Q_{1}(0) + S) \omega_{\ell_{0}}(0) \}
\end{eqnarray*}

\noindent
is a basis for $\Ker R_{S}(0)$, which shows that
 $(Q_{1}(0) + S)$ maps $\Ker R_{\DN}(0)$ onto $\Ker R_{S}(0)$.
We express $(Q_{1}(0) + S)^{-1}|_{\Ker R_{S}(0)}$ by an $\ell_{0} \times \ell_{0}$ matrix with respect to the bases
$\{ \varphi_{1}(0), \cdots, \varphi_{\ell_{0}}(0) \}$ and $\{ \omega_{1}(0), \cdots, \omega_{\ell_{0}}(0) \}$ as follows.
We put

\begin{eqnarray}   \label{E:3.5}
(Q_{1}(0) + S)^{-1} \varphi_{i}(0) & = & \sum_{k=1}^{\ell_{0}} a_{ik}(0) \omega_{k}(0), \qquad
{\mathcal A} ~ = ~ \left( a_{ik}(0) \right)_{1 \leq i, k \leq \ell_{0}}.
\end{eqnarray}

\noindent
In other words, $\varphi_{i}(0)  =  \sum_{k=1}^{\ell_{0}} a_{ik}(0) (Q_{1}(0) + S) \omega_{k}(0)$.
For each $i, k$, we extend $a_{ik}(0)$ to a smooth function $a_{ik}(\lambda)$ such that

\begin{eqnarray*}
\varphi_{i}(\lambda) & = & \sum_{k=1}^{\ell_{0}} a_{ik}(\lambda) (Q_{1}(\lambda) + S) \omega_{k}(\lambda) + \phi_{i}(\lambda),
\end{eqnarray*}

\noindent
where $\phi_{i}(\lambda) \in \bigg[ \Span \big\{ (Q_{1}(\lambda) + S) \omega_{1}(\lambda), \cdots, (Q_{1}(\lambda) + S) \omega_{\ell_{0}}(\lambda) \big\} \bigg]^{\perp}$
with $\phi_{i}(0) = 0$.
Then, we have

\begin{eqnarray}  \label{E:3.6}
 R_{S}(\lambda) \varphi_{i}(\lambda)
& = & ~ \sum_{k=1}^{\ell_{0}} ~ a_{ik}(\lambda) ~ \tau_{k}(\lambda) ~ \left( Q_{2}(\lambda) - S \right)^{-1} \omega_{k}(\lambda) + R_{S}(\lambda) \phi_{i}(\lambda).
\end{eqnarray}

\noindent
Since $\kappa_{j}(\lambda)$ and $\tau_{j}(\lambda)$ are also holomorphic with $\kappa_{j}(0) = \tau_{j}(0) = 0$ for $1 \leq j \leq \ell_{0}$,
it follows that $\lim_{\lambda \rightarrow 0} \frac{\kappa_{j}(\lambda)}{\lambda} = \kappa_{j}^{\prime}(0)$ and
$\lim_{\lambda \rightarrow 0} \frac{\tau_{j}(\lambda)}{\lambda} = \tau_{j}^{\prime}(0)$.
We also note that

\begin{eqnarray*}
\lim_{\lambda \rightarrow 0} \frac{1}{\lambda} \langle R_{S}(\lambda) \phi_{i}(\lambda), \varphi_{j}(0) \rangle & = &
\lim_{\lambda \rightarrow 0} \bigg\langle \phi_{i}(\lambda), \frac{R_{S}({\overline \lambda}) - R_{S}(0)}{{\overline \lambda}} \varphi_{j}(0) \bigg\rangle   \\
& = & \lim_{\lambda \rightarrow 0} \bigg\langle \phi_{i}(\lambda), R_{S}^{\prime}(0) \varphi_{j}(0) \bigg\rangle  ~ = ~ 0,
\end{eqnarray*}

\noindent
where

\begin{eqnarray*}
R_{S}^{\prime}(0) & = & \frac{d}{d \lambda}\big|_{\lambda=0} R_{S}(\lambda)
~ = ~ - \sum_{j=1}^{2} \big(Q_{j}(0) + (-1)^{j+1} S \big)^{-1} Q_{j}^{\prime}(0) \big(Q_{j}(0) + (-1)^{j+1} S \big)^{-1}.
\end{eqnarray*}


\noindent
We use this fact to have

\begin{eqnarray}   \label{E:3.7}
& & \lim_{\lambda \rightarrow 0} \frac{\kappa_{i}(\lambda)}{\lambda} \delta_{ij} ~ = ~
\lim_{\lambda \rightarrow 0} \frac{1}{\lambda} \langle R_{S}(\lambda) \varphi_{i}(\lambda), ~ \varphi_{j}(0) \rangle  \\
& = &  \lim_{\lambda \rightarrow 0} \sum_{k,s=1}^{\ell_{0}} ~ \frac{\tau_{k}(\lambda)}{\lambda} ~
a_{ik}(\lambda) \overline{a_{js}(0)}
~ \big\langle \left( Q_{2}(\lambda) - S \right)^{-1} \omega_{k}(\lambda), ~ (Q_{1}(0) + S) \omega_{s}(0)\big\rangle.
\nonumber
\end{eqnarray}

\noindent
For an orthonormal basis $\{ \psi_{1}, \cdots, \psi_{\ell_{0}} \}$ for $\Ker \Delta_{M, B}$,
we define $\ell_{0} \times \ell_{0}$ matrices ${\mathcal B}$, ${\mathcal C}$ by

\begin{eqnarray}   \label{E:3.8}
{\mathcal B} = \left( b_{ij} \right)_{1 \leq i, j \leq \ell_{0}}, \quad b_{ij} ~ = ~ \langle \psi_{i}|_{{\mathcal N}}, ~ \omega_{j}(0) \rangle_{{\mathcal N}}, \qquad
{\mathcal C} = \left( c_{ij} \right)_{1 \leq i, j \leq \ell_{0}}, \quad c_{ij} ~ = ~ \langle \psi_{i}|_{{\mathcal N}}, ~ \psi_{j}|_{{\mathcal N}} \rangle_{{\mathcal N}}.
\end{eqnarray}

\noindent
Since $\{ \psi_{1}|_{{\mathcal N}}, \cdots, \psi_{\ell_{0}}|_{{\mathcal N}}\}$ is a basis for $\Ker R_{\DN}(0)$ by (\ref{E:2.25}), it is straightforward that ${\mathcal C} = {\mathcal B} \overline{{\mathcal B}^{T}}$.
It is known (\cite{Le1}, \cite{MM}) that

\begin{eqnarray}   \label{E:3.9}
& & \lim_{\lambda \rightarrow 0} \frac{\lambda}{\tau_{j}(\lambda)} ~ = ~ (\overline{{\mathcal B}^{T}}{\mathcal B})_{jj} ~ \neq ~ 0,
\qquad  (\overline{{\mathcal B}^{T}}{\mathcal B})_{jk} ~ = ~ 0 \qquad \text{for} \quad j \neq k , \\
& & \lim_{\lambda \rightarrow 0} \Pi_{j=1}^{\ell_{0}} \frac{\lambda}{\tau_{j}(\lambda)} ~ = ~
\Pi_{j=1}^{\ell_{0}} (\overline{{\mathcal B}^{T}}{\mathcal B})_{jj} ~ = ~ \ddet \overline{{\mathcal B}^{T}}{\mathcal B}
 ~ = ~ \ddet {\mathcal C},  \nonumber
\end{eqnarray}

\noindent
which leads to

\begin{eqnarray*}
\lim_{\lambda \rightarrow 0} \frac{\kappa_{i}(\lambda)}{\lambda} \delta_{ij} & = &
 \sum_{k, s =1}^{\ell_{0}} a_{ik}(0) \overline{a_{js}(0)} \frac{1}{(\overline{{\mathcal B}^{T}}{\mathcal B})_{kk}}
\langle  \left( Q_{2}(0) - S \right)^{-1} \omega_{k}(0), ~ (Q_{1}(0) + S) \omega_{s}(0) \rangle.
\end{eqnarray*}

\noindent
Since $\omega_{j}(0)$ belongs to $\Ker R_{\DN}(0)$, we have $Q_{1}(0) \omega_{j}(0) = - Q_{2}(0) \omega_{j}(0)$, and hence
$(Q_{1}(0) + S) \omega_{j}(0) = - (Q_{2}(0) - S) \omega_{j}(0)$.
Since $Q_{1}(0)$ and $Q_{2}(0)$ are self-adjoint operators and $\overline{{\mathcal B}^{T}} {\mathcal B}$ is a diagonal matrix, we have

\begin{eqnarray}    \label{E:3.10}
\lim_{\lambda \rightarrow 0} \frac{\kappa_{i}(\lambda)}{\lambda} \delta_{ij}  =
 (-1) \sum_{k, s =1}^{\ell_{0}} a_{ik}(0) \overline{a_{js}(0)}
\frac{1}{(\overline{{\mathcal B}^{T}}{\mathcal B})_{kk}}
\langle \omega_{k}(0), ~ \omega_{s}(0) \rangle
~ = ~ - \big( ~ {\mathcal A} (\overline{{\mathcal B}^{T}}{\mathcal B})^{-1} \overline{{\mathcal A}^{T}} ~ \big)_{ij}.
\end{eqnarray}

\noindent

\noindent
Hence, we obtain

\begin{eqnarray}   \label{E:3.11}
\lim_{\lambda \rightarrow 0} \frac{\Pi_{j} \kappa_{j}(\lambda)}{\lambda^{\ell_{0}}}
& = & (-1)^{\ell_{0}} \frac{\ddet {\mathcal A} \overline{{\mathcal A}^{T}}}{\ddet {\mathcal C}},
\end{eqnarray}

\noindent
which leads to the following result.

\begin{theorem}   \label{Theorem:3.1}
We assume that $M_{0}$ is disconnected so that $M_{0} = M_{1} \cup M_{2}$ and assume that
both $\Delta_{M_{1}, {\frak R}_{1}, B}$ and $\Delta_{M_{2}, {\frak R}_{2}, B}$ are invertible operators with
$\ell_{0} = \Dim \Ker \Delta_{M, B}$.
We also take the line $\{ re^{i \theta_{0}} \mid - \frac{3 \pi}{2} < \theta_{0} < - \pi \}$ as a branch cut to define the logarithm.
Then,
\begin{eqnarray*}
& & \ln \Det^{\ast} \Delta_{M, B} - \ln \Det \Delta_{M_{1}, {\frak R}_{1}, B} - \ln \Det \Delta_{M_{2}, {\frak R}_{2}, B}   \\
& = & a_{0} + \ln (-1)^{\ell_{0}}
~ - ~ \ln \ddet {\mathcal C} + \ln \ddet {\mathcal A} \overline{{\mathcal A}^{T}}
~ + ~  \ln \Det^{\ast} R_{S}(0) \qquad (\Mod ~ 2\pi i),
\end{eqnarray*}
where $-a_{0}$ is the constant term in the asymptotic expansion of $\ln \Det R_{S}(\lambda)$ for $|\lambda| \rightarrow \infty$.
If $\Dim M$ is even, then $a_{0} = 0$.
\end{theorem}

\vspace{0.2 cm}

As a special case of Theorem \ref{Theorem:3.1}, we consider the scalar Laplacian $\Delta_{M}$ acting on smooth functions on a closed Riemannian manifold $M$. We assume that $\Delta_{M_{i}, {\frak R}_{i}, B}$ are invertible,
{\it i.e.} $Q_{1}(0) + S$ and $Q_{2}(0) - S$ are invertible.
Then $\ell_{0} = 1$ and ${\mathcal C}$ is a $1 \times 1$ matrix. Moreover,
$\psi_{1} = \frac{1}{\sqrt{\vol(M)}}$ is a constant function, and hence

\begin{eqnarray}  \label{E:3.12}
{\mathcal C} & = & \bigg\langle \frac{1}{\sqrt{\vol(M)}}\big|_{{\mathcal N}}, ~  \frac{1}{\sqrt{\vol(M)}}\big|_{{\mathcal N}} \bigg\rangle ~ = ~ \frac{\vol({\mathcal N})}{\vol(M)}.
\end{eqnarray}

\noindent
Here $\Ker R_{S}(0) = \{ S c \mid c ~ \text{is a constant function}\}$ and $\Ker R_{\DN}(0) = \ker Q_{i}(0) = \{ \text{constant functions}\}$.
Eq.(\ref{E:3.5}) shows that ${\mathcal A} = \left( \frac{\sqrt{\vol({\mathcal N})}}{\parallel S(1)\parallel} \right)$.
Then, Theorem \ref{Theorem:3.1} is rewritten as follows.

\begin{corollary}   \label{Corollary:3.2}
We make the same assumptions as in Theorem \ref{Theorem:3.1}.
Let $\Delta_{M}$ be the scalar Laplacian acting on smooth functions on a closed Riemannian manifold $M$.
Then,
\begin{eqnarray*}
& & \ln \Det^{\ast} \Delta_{M, B} - \ln \Det \Delta_{M_{1}, {\frak R}_{1}, B} - \ln \Det \Delta_{M_{2}, {\frak R}_{2}, B}  \\
 & = & a_{0} ~ + ~ \ln (-1) ~ - ~ 2 \ln \parallel S(1)\parallel  ~ + ~ \ln \vol(M) ~ + ~ \ln \Det^{\ast} R_{S}(0) \qquad (\Mod ~ 2\pi i).
\end{eqnarray*}
\end{corollary}

\vspace{0.2 cm}
Finally, we discuss the conformal invariance of $\frac{\Det^{\ast} R_{S}(0)}{\parallel S(1) \parallel^{2}}$ when $\Dim M = 2$. Let $(M, g)$ be a $2$-dimensional closed Riemannian manifold and $\Delta_{M}$ be a scalar Laplacian.
We denote a conformal change of a metric by $g_{ij}(\epsilon) = e^{2\epsilon F} g_{ij}$, where $F : M \rightarrow {\mathbb R}$ is a smooth function.
We denote by $\Delta_{M}(\epsilon) := e^{- 2 \epsilon F} \Delta_{M}$ the Laplacian with respect to the metric $g_{ij}(\epsilon)$.
We assume that $\Delta_{M_{1}, {\frak R}_{1}}$ and $\Delta_{M_{2}, {\frak R}_{2}}$ are invertible. Then, Corollary \ref{Corollary:3.2} is rewritten by

\begin{eqnarray}    \label{E:3.13}
& & \ln \frac{\Det^{\ast} R_{S}(0)}{\parallel S(1)\parallel^{2}} (\epsilon)  \\
& = & \ln \Det^{\ast} \Delta_{M}(\epsilon) - \ln \Det \Delta_{M_{1}, {\frak R}_{1}}(\epsilon) - \ln \Det \Delta_{M_{2}, {\frak R}_{2}}(\epsilon) ~ - ~ \ln \vol(M)(\epsilon) ~ - ~ \ln (-1),  \quad (\Mod ~ 2\pi i).    \nonumber
\end{eqnarray}

\noindent
We denote

\begin{eqnarray}    \label{E:3.14}
\Tr \big( F e^{- t \Delta_{M}} \big) ~ \sim ~ \sum_{k=0}^{\infty} \beta_{k}(F) ~ t^{\frac{-2 + k}{2}}, \qquad
\Tr \big( F e^{- t \Delta_{M_{i}, {\frak R}_{i}}} \big) ~ \sim ~ \sum_{k=0}^{\infty} \beta^{(i)}_{k}(F) ~ t^{\frac{-2 + k}{2}}.
\end{eqnarray}

\noindent
By Theorem 2.12 in \cite{BG2} and Theorem 3.3.1, Theorem 3.5.1 in \cite{Gi2} (or Chapter 4 in \cite{Ki}), it follows that

\begin{eqnarray}     \label{E:3.15}
\frac{d}{d\epsilon}\big|_{\epsilon=0} \ln \frac{\Det^{\ast} R_{S}(0)}{\parallel S(1)\parallel^{2}} (\epsilon) & = &
- 2 \bigg( \beta_{2}(F) - \frac{1}{\vol(M)} \int_{M} F dx - \beta_{2}^{(1)}(F) - \beta_{2}^{(2)}(F) \bigg) - \frac{2}{\vol(M)} \int_{M} F dx  \nonumber \\
& = & 0,
\end{eqnarray}

\noindent
which we used the fact that $\beta_{2}(F) - \beta^{(1)}_{2}(F) - \beta_{2}^{(2)}(F) = 0$ since the Robin boundary condition on $M_{1}$ is $\nabla_{\partial_{{\mathcal N}}} + S$ and on $M_{2}$ is $\nabla_{- \partial_{{\mathcal N}}} - S$. This is an analogue of Theorem 1.1 in \cite{GG} and \cite{EW}.

\begin{corollary}   \label{Corollary:3.3}
Let $(M, g)$ be a $2$-dimensional closed Riemannian manifold and $\Delta_{M}$ be a scalar Laplacian.
We choose a closed hypersurface ${\mathcal N}$ such that the closure of $M - {\mathcal N}$ is a union of $M_{1}$ and $M_{2}$. We assume that $\Delta_{M_{1}, {\frak R}_{1}}$ and $\Delta_{M_{2}, {\frak R}_{2}}$ are invertible. Then, $\frac{\Det^{\ast} R_{S}(0)}{\parallel S(1) \parallel^{2}}$ is a conformal invariant.
\end{corollary}

\vspace{0.3 cm}

\section{Comparison of the Robin and Dirichlet Boundary Conditions}
Let $(W, Y \cup Z, g^{W})$ be an $m$-dimensional compact oriented Riemannian manifold with boundary $Y \cup Z$, where
$Y \cap Z = \emptyset$ and $Z$ may be empty. If $Z \neq \emptyset$, we impose the Dirichlet or Neumann boundary condition $B$ on Z as before.
We let ${\mathcal E} \rightarrow W$ be a Hermitian vector bundle on $W$ and $\Delta_{W}$ be a Laplacian
such that $\Delta_{W}$ satisfies (\ref{E:2.1}), (\ref{E:2.2}) and that
$\Delta_{W, D, B}$ is a non-negative operator, where $\Delta_{W, D, B}$ is the
Laplacian $\Delta_{W}$ with the Dirichlet condition $D$ on $Y$ and $B$ on $Z$.
Then, $\Delta_{W, D, B}$ is a symmetric operator with the principal symbol
$\sigma_{L}(\Delta_{W})(x^{\prime}, \xi^{\prime}) = \parallel \xi^{\prime} \parallel^{2}$.
We choose an outward unit normal vector field $\partial_{x_{m}}$ on $Y$ and choose the Robin boundary condition ${\frak R}$ on $Y$ defined by

\begin{eqnarray}   \label{E:4.1}
{\frak R} : C^{\infty}({\mathcal E}) \rightarrow C^{\infty}({\mathcal E}|_{Y}), \qquad  {\frak R} \phi ~ = ~ \left( \nabla_{\partial_{x_{m}}} \phi \right)|_{Y} + S \left( \phi|_{Y} \right),
\end{eqnarray}

\noindent
where $S : Y \rightarrow {\mathcal End} \left({\mathcal E}|_{Y} \right)$ is a section such that
$S(y) : {\mathcal E}_{y} \rightarrow {\mathcal E}_{y}$ is symmetric for each $y \in Y$ as in
Definition \ref{Definition:2.1}.
In this section, we are going to discuss
$~~\ln \Det \Delta_{W, {\frak R}, B} - \ln \Det \Delta_{W, D, B}$.
We define the Poisson operator
${\mathcal P}_{D}(\lambda) : C^{\infty}({\mathcal E}|_{Y}) \rightarrow C^{\infty}({\mathcal E})$ as follows.
For $f \in C^{\infty}({\mathcal E}|_{Y})$, let ${\widetilde f}$ be any smooth extension of $f$ satisfying
$B({\widetilde f}) = 0$ on $Z$. Then,
${\mathcal P}_{D}(\lambda) f$ is defined by

\begin{eqnarray*}
{\mathcal P}_{D}(\lambda) f & = & {\widetilde f} -
\big( \Delta_{W, D, B} + \lambda \big)^{-1} \big( \Delta_{W} + \lambda \big) {\widetilde f}.
\end{eqnarray*}

\noindent
Then, ${\mathcal P}_{D}(\lambda)$ satisfies

\begin{eqnarray}   \label{E:4.2}
(\Delta_{W} + \lambda) \cdot {\mathcal P}_{D}(\lambda) = 0, \qquad
\gamma_{0} \cdot {\mathcal P}_{D}(\lambda) = \Id_{C^{\infty}({\mathcal E}|_{Y})}, \qquad
B \cdot {\mathcal P}_{D}(\lambda) = 0,
\end{eqnarray}

\noindent
where $\gamma_{0}$ is the restriction map to $Y$.
Taking derivative with respect to $\lambda$, we get

\begin{eqnarray}   \label{E:4.3}
\frac{d}{d \lambda} {\mathcal P}_{D}(\lambda) & = &
- (\Delta_{W, D, B} + \lambda)^{-1} \cdot {\mathcal P}_{D}(\lambda).
\end{eqnarray}

\vspace{0.2 cm}

\begin{definition}   \label{Definition:4.1}
We define $Q_{D}(\lambda), ~ Q_{D}(\lambda) + S : C^{\infty}({\mathcal E}|_{Y}) \rightarrow C^{\infty}({\mathcal E}|_{Y})$ as follows.
\begin{eqnarray*}
& & Q_{D}(\lambda) f := \big( \nabla_{\partial_{x_{m}}} {\mathcal P}_{D}(\lambda) f \big)\big|_{Y},  \\
& & ( Q_{D}(\lambda) + S ) f := \big( \nabla_{\partial_{x_{m}}} {\mathcal P}_{D}(\lambda) f \big)\big|_{Y} + S f ~ = ~
{\frak R} \big({\mathcal P}_{D}(\lambda) f \big).
\end{eqnarray*}
\end{definition}

\noindent
From the definition it follows that

\begin{eqnarray}     \label{E:4.4}
\Ker Q_{D}(0) & = & \{ \phi|_{Y} \mid \phi \in \Ker \Delta_{M, N, B} \}, \qquad
\Ker ( Q_{D}(0) + S ) ~ = ~ \{ \phi|_{Y} \mid \phi \in \Ker \Delta_{M, {\frak R}, B} \}.
\end{eqnarray}

\noindent
We note that

\begin{eqnarray}   \label{E:4.5}
\frac{d}{d \lambda} ( Q_{D}(\lambda) + S ) & = &
{\frak R} \cdot \frac{d}{d \lambda} {\mathcal P}_{D}(\lambda)
~ = ~ {\frak R} \cdot
\left\{ - (\Delta_{W, D, B} + \lambda)^{-1} \cdot {\mathcal P}_{D}(\lambda)  \right\}    \\
& = & {\frak R} \cdot \left( (\Delta_{W, {\frak R}, B} + \lambda)^{-1} - (\Delta_{W, D, B} + \lambda)^{-1} \right) \cdot {\mathcal P}_{D}(\lambda)   \nonumber  \\
& = & {\frak R} \cdot {\mathcal P}_{D}(\lambda) \cdot \gamma_{0} \cdot
(\Delta_{W, {\frak R}, B} + \lambda)^{-1} \cdot  {\mathcal P}_{D}(\lambda)    \nonumber  \\
& = & ( Q_{D}(\lambda) + S ) \cdot \gamma_{0} \cdot
(\Delta_{W, {\frak R}, B} + \lambda)^{-1} \cdot  {\mathcal P}_{D}(\lambda),   \nonumber
\end{eqnarray}

\noindent
which shows that

\begin{eqnarray}   \label{E:4.6}
( Q_{D}(\lambda) + S )^{-1} \frac{d}{d \lambda} ( Q_{D}(\lambda) + S ) & = &
\gamma_{0} \cdot (\Delta_{W, {\frak R}, B} + \lambda)^{-1} \cdot  {\mathcal P}_{D}(\lambda).
\end{eqnarray}

\noindent
We now put $\nu = \big[\frac{m-1}{2}\big] + 1$. Then,

\begin{eqnarray}    \label{E:4.7}
& & \frac{d^{\nu}}{d \lambda^{\nu}} \left\{ \ln \Det \left( \Delta_{W, {\frak R}, B} + \lambda \right) -
\ln \Det \left( \Delta_{W, D, B} + \lambda \right) \right\}   \\
& = & \Tr \left\{ \frac{d^{\nu-1}}{d \lambda^{\nu-1}} \left( (\Delta_{W, {\frak R}, B} + \lambda)^{-1} -
(\Delta_{W, D, B} + \lambda)^{-1} \right) \right\}    \nonumber  \\
& = & \Tr \left\{ \frac{d^{\nu-1}}{d \lambda^{\nu-1}} \left( {\mathcal P}_{D}(\lambda) \cdot \gamma_{0}
 \cdot (\Delta_{W, {\frak R}, B} + \lambda)^{-1} \right) \right\}
~ = ~ \Tr \left\{ \frac{d^{\nu-1}}{d \lambda^{\nu-1}} \left(  \gamma_{0} \cdot
(\Delta_{W, {\frak R}, B} + \lambda)^{-1} \cdot {\mathcal P}_{D}(\lambda) \right) \right\},    \nonumber
\end{eqnarray}

\noindent
which leads to the following result.

\begin{eqnarray}  \label{E:4.8}
\frac{d^{\nu}}{d \lambda^{\nu}} \left\{ \ln \Det \left( \Delta_{W, {\frak R}, B} + \lambda \right) -
\ln \Det \left( \Delta_{W, D, B} + \lambda \right) \right\}
& = & \frac{d^{\nu}}{d \lambda^{\nu}} \ln \Det ( Q_{D}(\lambda) + S ).
\end{eqnarray}

\noindent
For some polynomial $G(\lambda) = \sum_{j=0}^{\frac{[m-1]}{2}} b_{j} \lambda^{j}$, it follows (cf. \cite{KL5}) that

\begin{eqnarray}    \label{E:4.9}
\ln \Det (\Delta_{W, {\frak R}, B} + \lambda) - \ln \Det (\Delta_{W, D, B} + \lambda) & = & \sum_{j=0}^{\frac{[m-1]}{2}} b_{j} \lambda^{j} + \ln \Det (Q_{D}(\lambda) + S).
\end{eqnarray}

\noindent
We note that both terms in the left hand side of (\ref{E:4.9}) have asymptotic expansions for $|\lambda| \rightarrow \infty$ whose constant terms are zero. Comparison of both sides of (\ref{E:4.9}) implies that
$\ln \Det (Q_{D}(\lambda) + S)$ has an asymptotic expansion for $|\lambda| \rightarrow \infty$.
Hence, $- b_{0}$ is the constant term in the asymptotic expansion of $\ln \Det (Q_{D}(\lambda) + S)$
for $|\lambda| \rightarrow \infty$.

We next consider the asymptotic behaviors for $\lambda \rightarrow 0$.
Let $p_{0} = \Dim \Ker \Delta_{W, {\frak R}, B} = \Dim \Ker (Q_{D}(0) + S)$.
We denote by $\{ \eta_{k}(\lambda) \mid k \in {\mathbb N} \}$ and $\{ \phi_{k}(\lambda) \mid k \in {\mathbb N} \}$ the eigenvalues and corresponding orthonormal eigensections of $Q_{D}(\lambda) + S$, where $\eta_{j}(0) = 0$ for $1 \leq j \leq p_{0}$. Then,

\begin{eqnarray}    \label{E:4.10}
& & \ln \Det^{\ast}\Delta_{W, {\frak R}, B} ~ - ~ \ln \Det \Delta_{W, D, B}   \\
& = & \lim_{\lambda \rightarrow 0} \ln \frac{\eta_{1}(\lambda) \cdots \eta_{p_{0}}(\lambda)}{\lambda^{p_{0}}}
~ + ~ b_{0} ~ + ~ \ln \Det^{\ast} (Q_{D}(0) + S) \qquad (\Mod  2 \pi i).   \nonumber
\end{eqnarray}

\noindent
We choose $\Phi_{k}(\lambda) \in C^{\infty}(W)$ such that

\begin{eqnarray}   \label{E:4.11}
(\Delta_{W} + \lambda) \Phi_{k}(\lambda) = 0, \qquad \Phi_{k}(\lambda)|_{Y} = \phi_{k}(\lambda), \qquad B(\Phi_{k}(\lambda)) = 0.
\end{eqnarray}

\noindent
Using the Green theorem (Lemma 1.4.17 in \cite{Gi2}), it follows that

\begin{eqnarray}   \label{E:4.12}
& & 0 ~ = ~ \langle (\Delta_{W} + \lambda) \Phi_{i}(\lambda), \Phi_{j}(0) \rangle_{W} ~ = ~
\langle \Delta_{W} \Phi_{i}(\lambda), \Phi_{j}(0) \rangle_{W} ~ + ~ \lambda ~ \langle  \Phi_{i}(\lambda), \Phi_{j}(0) \rangle_{W}  \\
& = & \langle \Delta_{W} \Phi_{i}(\lambda), \Phi_{j}(0) \rangle_{W} ~ - ~
\langle \Phi_{i}(\lambda), \Delta_{W} \Phi_{j}(0) \rangle_{W} ~ + ~
\lambda ~ \langle  \Phi_{i}(\lambda), \Phi_{j}(0) \rangle_{W}       \nonumber  \\
& = & - \langle \nabla_{\partial_{x_{m}}} \phi_{i}(\lambda), \phi_{j}(0) \rangle_{Y} ~ + ~
\langle \phi_{i}(\lambda), \nabla_{\partial_{x_{m}}} \phi_{j}(0) \rangle_{Y} ~ + ~
\lambda ~ \langle  \Phi_{i}(\lambda), \Phi_{j}(0) \rangle_{W}       \nonumber \\
& = & - \langle (Q_{D}(\lambda) + S) \phi_{i}(\lambda), \phi_{j}(0) \rangle_{Y} ~ + ~
\langle \phi_{i}(\lambda), (Q_{D}(0) + S) \phi_{j}(0) \rangle_{Y} ~ + ~
\lambda ~ \langle  \Phi_{i}(\lambda), \Phi_{j}(0) \rangle_{W}     \nonumber \\
& = & - \eta_{i}(\lambda) ~ \langle \phi_{i}(\lambda), \phi_{j}(0) \rangle_{Y} ~ + ~
\lambda ~ \langle  \Phi_{i}(\lambda), \Phi_{j}(0) \rangle_{W},    \nonumber
\end{eqnarray}

\noindent
which shows that

\begin{eqnarray}       \label{E:4.13}
\lim_{\lambda \rightarrow 0} \frac{\eta_{i}(\lambda)}{\lambda} \delta_{ij} & = & \langle  \Phi_{i}(0), \Phi_{j}(0) \rangle_{W}.
\end{eqnarray}

\noindent
We denote by $\{ \psi_{1}, \cdots, \psi_{p_{0}} \}$ an orthonormal basis of $\Ker \Delta_{W, {\frak R}, B}$
and define three $p_{0} \times p_{0}$ matrices by

\begin{eqnarray}    \label{E:4.14}
& & a_{ij} = \langle \psi_{i}|_{Y}, \psi_{j}|_{Y} \rangle_{Y}, \qquad {\frak A} = \left( a_{ij} \right), \qquad
b_{ij} = \langle \psi_{i}|_{Y}, \phi_{j}(0) \rangle_{Y},  \qquad {\frak B} = \left( b_{ij} \right),  \\
& & c_{ij} = \langle  \Phi_{i}(0), \Phi_{j}(0) \rangle_{W}, \qquad {\frak C} = \left( c_{ij} \right).  \nonumber
\end{eqnarray}

\noindent
Then, it follows that $\psi_{i}|_{Y} = \sum_{k} b_{ik} \phi_{k}(0)$
and $\psi_{i} = \sum_{k} b_{ik} \Phi_{k}(0)$, which shows that

\begin{eqnarray}    \label{E:4.15}
{\frak A} & = & {\frak B} \overline{{\frak B}^{T}}, \qquad
\Id ~ = ~ {\frak B} {\frak C} \overline{{\frak B}^{T}}.
\end{eqnarray}

\noindent
Hence, we have

\begin{eqnarray}     \label{E:4.16}
\lim_{\lambda \rightarrow 0} \ln \frac{\eta_{1}(\lambda) \cdots \eta_{p_{0}}(\lambda)}{\lambda^{p_{0}}} & = &
\ln \ddet {\frak C} ~ = ~ - \ln \ddet {\frak B} \overline{{\frak B}^{T}} ~ = ~
- \ln \ddet {\frak A},
\end{eqnarray}

\noindent
which leads to the following result.

\begin{theorem}  \label{Theorem:4.2}
Let $(W, Y \cup Z, g^{W})$ be a compact oriented Riemannian manifold with boundary $Y \cup Z$
with boundary conditions ${\frak R}$ and $D$ on $Y$ and $B$, respectively. Then,
the following equality holds.
\begin{eqnarray*}
\ln \Det^{\ast} \Delta_{W, {\frak R}, B}  - \ln \Det \Delta_{W, D, B} & = &
b_{0} - \ln \ddet {\frak A} + \ln \Det^{\ast} \left( Q_{D}(0) + S \right) \qquad (\Mod  2 \pi i),
\end{eqnarray*}
where $- b_{0}$ is the constant term in the asymptotic expansion of $\ln \Det (Q_{D}(\lambda) + S)$ for $|\lambda| \rightarrow \infty$.
If $Q_{D}(0) + S$ is an invertible operator, $\ln \ddet {\frak A}$ does not appear.
\end{theorem}

\begin{corollary}  \label{Corollary:4.3}
Let ${\frak R}_{p} = \nabla_{\partial_{x_{m}}} + S_{p}$ and ${\frak R}_{q} = \nabla_{\partial_{x_{m}}} + S_{q}$ be two Robin boundary conditions, where $S_{p}(y)$ and $S_{q}(y)$ are symmetric for $y \in Y$.
Then,
\begin{eqnarray*}
& & \ln \Det^{\ast} \Delta_{W, {\frak R}_{p}, B}  - \ln \Det^{\ast} \Delta_{W, {\frak R}_{q}, B}  \\
& = & b^{p}_{0} - b^{q}_{0} - \ln \ddet {\frak A}_{p} + \ln \ddet {\frak A}_{q} +
\ln \Det^{\ast} \left( Q_{D}(0) + S_{p} \right) - \ln \Det^{\ast} \left( Q_{D}(0) + S_{q} \right)
\qquad (\Mod  2 \pi i).
\end{eqnarray*}
\end{corollary}

\vspace{0.2 cm}
We next compute $b_{0}$ in Theorem \ref{Theorem:4.2} in some specific case. We suppose that $Y$ consists of $k_{0}$ components, {\it i.e.}
$Y = \cup_{i=1}^{k_{0}} Y_{i}$, and assume that the metric on $W$ is a product one near $Y$ so that $\Delta_{W} = - \partial_{x_{m}}^{2} + \Delta_{Y}$ on a collar neighborhood of $Y$.
Then, it is known that $Q_{D}(\lambda)$ differs from $(\sqrt{\Delta_{Y} + \lambda})$ by a smoothing operator (\cite{Le2}, \cite{PW}). We choose $S = \left( \alpha_{i} \delta_{ij} \right)_{1 \leq i, j \leq k_{0}}$.
Hence, for some smoothing operator ${\mathcal K}(\lambda)$ on $Y$,

\begin{eqnarray}    \label{E:4.17}
Q_{D}(\lambda) + S & = & \left( \begin{array}{clcr} \sqrt{\Delta_{Y_{1}} + \lambda} + \alpha_{1} & 0 & 0  \\ 0 & \ddots & 0 \\ 0 & 0 & \sqrt{\Delta_{Y_{k_{0}}} + \lambda} + \alpha_{k_{0}} \end{array} \right)  ~ + ~ {\mathcal K}(\lambda).
\end{eqnarray}

\noindent
It is shown in the appendix of \cite{BFK1} that $\ln \Det (Q_{D}(\lambda) + \alpha)$ and
$\sum_{i=1}^{k_{0}} \ln \Det \big( \sqrt{\Delta_{Y_{i}} + \lambda} + \alpha_{i} \big)$ have the same asymptotic expansion for $|\lambda| \rightarrow \infty$.
The constant term $b^{i}_{0}$ in the asymptotic expansion of $- \ln \Det \left( \sqrt{\Delta_{Y_{i}} + \lambda} + \alpha_{i} \right)$
is given in (\ref{E:2.58}), which is $b^{i}_{0} = - {\frak s}_{\alpha_{i}}$.

\begin{corollary}  \label{Corollary:4.5}
We assume that $\Dim W \geq 2$ and the metric is a product one near the boundary $Y$ so that $\Delta_{W} = - \partial_{x_{m}}^{2} + \Delta_{Y}$ on a collar neighborhood of $Y$. Suppose that $Y$ consists of $k_{0}$ components and choose $S$ by (\ref{E:2.78}) so that $Q_{D}(\lambda) + S$ is given by (\ref{E:4.17}). Then,
\begin{eqnarray*}
\ln \Det^{\ast} \Delta_{W, {\frak R}, B}  - \ln \Det \Delta_{W, D, B} & = &
- \sum_{i=1}^{k_{0}} {\frak s}_{\alpha_{i}} - \ln \ddet {\frak A}  + \ln \Det^{\ast} \left( Q_{D}(0) + S \right)
 \qquad (\Mod  2 \pi i),
\end{eqnarray*}
where ${\frak s}_{\alpha_{i}}$ is given by (\ref{E:2.58}). In particular, if $\alpha_{i} = 0$, then ${\frak s}_{\alpha_{i}} = 0$.
If $Q_{D}(0) + S$ is invertible, then, $\ln \ddet {\frak A}$ does not appear.
\end{corollary}

\vspace{0.2 cm}

As an application of Corollary \ref{Corollary:4.5},
we are going to compute the zeta-determinant of a Laplacian on a cylinder with the Robin boundary condition
${\frak R}_{\alpha} := \partial_{x_{m}} + \alpha \Id$, where $\partial_{x_{m}}$ is the outward unit vector field on each boundary and $Z = \emptyset$. For simple presentation we assume that $\alpha > 0$.
We first consider the case that $W = [0, L]$, a line segment of length $L > 0$ with
$\Delta_{W} = - \frac{d^{2}}{du^{2}}$. One can compute that

\begin{eqnarray}    \label{E:4.18}
\ln \Det \Delta_{W, D, D} & = & \ln 2L.
\end{eqnarray}

\noindent
Simple computation shows that
$Q_{D}(\lambda) + \alpha : C^{\infty}(\{ 0 \}) \oplus C^{\infty}(\{ L \}) \rightarrow C^{\infty}(\{ 0 \}) \oplus C^{\infty}(\{ L \})$ and its determinant are given by

\begin{eqnarray}     \label{E:4.19}
& & Q_{D}(\lambda) + \alpha ~ = ~ \left( \begin{array}{clcr}
\frac{\sqrt{\lambda} ( e^{2\sqrt{\lambda}L} + 1)}{e^{2\sqrt{\lambda}L} - 1} + \alpha &
\frac{- 2 \sqrt{\lambda}e^{\sqrt{\lambda}L}}{e^{2 \sqrt{\lambda}L} - 1} \\
\frac{- 2 \sqrt{\lambda} e^{\sqrt{\lambda}L}}{e^{2\sqrt{\lambda}L} - 1} &
\frac{\sqrt{\lambda} ( e^{2\sqrt{\lambda}L} + 1)}{e^{2\sqrt{\lambda}L} - 1} + \alpha
\end{array} \right), \\
& & \ddet \left( Q_{D}(\lambda) + \alpha \right) ~ = ~
\lambda + \alpha^{2} + 2 \alpha \sqrt{\lambda} + \frac{4 \alpha \sqrt{\lambda}}{e^{2\sqrt{\lambda}L} - 1},    \nonumber
\end{eqnarray}

\noindent
which shows that $\ln \ddet \left( Q_{D}(0) + \alpha \right) = \ln \left( \frac{2\alpha}{L} + \alpha^{2} \right)$ and $b_{0} = 0$.
Since $Q_{D}(0) + \alpha$ is invertible, $\ln \ddet {\frak A}$ does not appear.
Hence, Theorem \ref{Theorem:4.2} tells that

\begin{eqnarray}    \label{E:3.20}
\ln \Det \Delta_{W, {\frak R}_{\alpha}, {\frak R}_{\alpha}} & = & \ln 2L + \ln \left( \frac{2\alpha}{L} + \alpha^{2} \right) ~ = ~ \ln \left( 2 \alpha ( L \alpha + 2 ) \right),
\end{eqnarray}

\noindent
which is computed earlier in \cite{MKB} and also can be computed by using the formula given in \cite{BFK2}.

For a closed Riemannian manifold $Y$ of dimension $m-1$ with $m \geq 2$, we let $W = [0, L] \times Y$ with the product metric and $Z = \emptyset$. Let
$\Delta_{W} = - \partial_{u}^{2} + \Delta_{Y}$, where we assume that $\Delta_{Y}$ is a non-negative operator.
It is known (Proposition 5.1 in \cite{MM}) that

\begin{eqnarray}   \label{E:4.21}
\ln \Det \Delta_{W, D, D} & = & q_{0} \ln 2L - 2L (\ln 2 -1) \cdot \Res_{s=-\frac{1}{2}} \zeta_{\Delta_{Y}}(s) +
L \cdot \Fp_{s=-\frac{1}{2}} \zeta_{\Delta_{Y}}(s) \\
& &  - \frac{1}{2} \ln \Det^{\ast} \Delta_{Y}
~ + ~ \sum_{\mu_{j} > 0} \ln \left( 1 - e^{-2 L \sqrt{\mu_{j}}} \right),   \nonumber
\end{eqnarray}

\noindent
where $q_{0} = \Dim \Ker \Delta_{Y}$. We choose the Robin boundary condition ${\frak R}_{\alpha} := \partial_{x_{m}} + \alpha \Id$ for $\alpha \in {\mathbb R}$ such that $- \alpha \notin \Spec(Q_{D}(0)) \cup \Spec(\sqrt{\Delta_{Y}})$. Here we do not assume $\alpha > 0$.
Simple computation shows that the eigenvalues of $Q_{D}(0) + \alpha : C^{\infty}(\{ 0 \} \times Y) \oplus C^{\infty}(\{ L \} \times Y) \rightarrow
C^{\infty}(\{ 0 \} \times Y) \oplus C^{\infty}(\{ L \} \times Y)$ are given by

\begin{eqnarray}   \label{E:4.22}
& & \Spec \left( Q_{D}(0) + \alpha \right) \\
& = &  \left\{ \alpha, \frac{2}{L} + \alpha \right\}  \cup
\left\{ \sqrt{\mu_{j}} + \alpha - \frac{2 \sqrt{\mu_{j}}}{e^{L \sqrt{\mu_{j}}} + 1},
\sqrt{\mu_{j}} + \alpha + \frac{2 \sqrt{\mu_{j}}}{e^{L \sqrt{\mu_{j}}} - 1}  \mid
0 < \mu_{j} \in \Spec \left( \Delta_{Y} \right) ~ \right\},    \nonumber
\end{eqnarray}

\noindent
where the multiplicities of $\alpha$ and $\frac{2}{L} + \alpha$ are $q_{0}$. It follows that

\begin{eqnarray}    \label{E:4.23}
\ln \Det \left( Q_{D}(0) + \alpha \right)  & = & 2 \ln \Det \left( \sqrt{\Delta_{Y}} + \alpha \right) + q_{0} \ln \left( 1 + \frac{2}{L \alpha}  \right)  \\
&  & + ~ \sum_{\mu_{j} > 0} \ln \left( 1 + \frac{4 \alpha \sqrt{\mu_{j}}}{\left( \sqrt{\mu_{j}} + \alpha \right)^{2}
\left( e^{2L \sqrt{\mu_{j}}} - 1 \right)} \right) \qquad (\Mod ~ 2 \pi i).  \nonumber
\end{eqnarray}

\noindent
Since $Q_{D}(0) + \alpha$ is invertible, $\ln \ddet {\frak A}$ in
Corollary \ref{Corollary:4.5} does not appear. Moreover, $Q_{D}(\lambda) + \alpha$ differs from
$\left( \begin{array}{clcr} \sqrt{\Delta_{Y} + \lambda} + \alpha & 0 \\ 0 & \sqrt{\Delta_{Y} + \lambda} + \alpha \end{array} \right)$
by a smoothing operator,
which shows that the constant term in the asymptotic expansion of $\ln \Det ( Q_{D}(\lambda) + \alpha )$ is $ 2 {\frak s}_{\alpha}$
by (\ref{E:2.58}).
Corollary \ref{Corollary:4.5} leads to the following result.

\begin{theorem}   \label{Theorem:4.6}
Let $\alpha \in {\mathbb R}$ such that $- \alpha \notin \Spec(\sqrt{\Delta_{Y}}) \cup \Spec(Q_{D}(0))$ and $W = [0, L] \times Y$ with the product metric.
Then,
\begin{eqnarray*}
& & \ln \Det \Delta_{W, {\frak R}_{\alpha}, {\frak R}_{\alpha}} ~ = ~ - 2 {\frak s}_{\alpha} + q_{0} \ln 2 \left( L + \frac{2}{\alpha} \right) - 2L (\ln 2 -1) \cdot \Res_{s=-\frac{1}{2}} \zeta_{\Delta_{Y}}(s) + L \cdot \Fp_{s=-\frac{1}{2}} \zeta_{\Delta_{Y}}(s) \\
& & \hspace{1.0 cm} + ~ 2 \ln \Det \left( \sqrt{\Delta_{Y}} + \alpha \right) - \frac{1}{2} \ln \Det^{\ast} \Delta_{Y} + ~
\sum_{\mu_{j} > 0} \ln \left( 1  -
\frac{\left( \sqrt{\mu_{j}} - \alpha \right)^{2}}{\left( \sqrt{\mu_{j}} + \alpha \right)^{2} e^{2L \sqrt{\mu_{j}}}} \right) \qquad (\Mod ~ 2 \pi i).
\end{eqnarray*}
\end{theorem}

\vspace{0.2 cm}
\noindent
{\it Remark} : If $\alpha = 0$, it follows by (\ref{E:4.22}) that $\ln \Det^{\ast} Q_{D}(0) = q_{0} \ln \frac{2}{L} + \ln \Det^{\ast} \Delta_{Y}$.
Let $\{ w_{1}, \cdots, w_{q_{0}} \}$ be an orthonormal basis of $\Ker \Delta_{Y}$.
Then, $\{ \frac{1}{\sqrt{L}} w_{1}, \cdots, \frac{1}{\sqrt{L}} w_{q_{0}} \}$ is an orthonormal basis of $\Ker \Delta_{W, N, N}$.
Since $\partial W = (\{ 0 \} \times Y) \cup (\{ L \} \times Y)$, it follows that ${\frak A} = \frac{2}{L} \Id_{q_{0} \times q_{0}}$.
Since ${\frak s}_{\alpha} = 0$, Corollary \ref{Corollary:4.5} shows that

\begin{eqnarray}     \label{E:4.24}
\ln \Det^{\ast} \Delta_{W, N, N} & = & q_{0} \ln 2L - 2L (\ln 2 -1) \cdot \Res_{s=-\frac{1}{2}} \zeta_{\Delta_{Y}}(s) +
L \cdot \Fp_{s=-\frac{1}{2}} \zeta_{\Delta_{Y}}(s) \\
& &  + \frac{1}{2} \ln \Det^{\ast} \Delta_{Y}
~ + ~ \sum_{\mu_{j} > 0} \ln \left( 1 - e^{-2 L \sqrt{\mu_{j}}} \right).   \nonumber
\end{eqnarray}

\noindent
Since $\Delta_{W, N, N}$ is a non-negative operator, we have no ambiguity of an imaginary part in logarithm.

\vspace{0.2 cm}
We next compute $\ln \Det \Delta_{W, N, D}$, where $\Delta_{W, N, D}$ is the Laplacian on $W = [0, L] \times Y$ with the Neumann condition on $\{ 0 \} \times Y$ and
the Dirichlet condition on $\{ L \} \times Y$. Here $Z = \{ L \} \times Y$ with the Dirichlet condition in our setting.
Corollary \ref{Corollary:4.5} shows that

\begin{eqnarray}   \label{E:4.25}
\ln \Det \Delta_{W, N, D} & = & \Det \Delta_{W, D, D} + \ln \Det Q_{D}(0),
\end{eqnarray}

\noindent
where $Q_{D}(0) : C^{\infty}(\{ 0 \} \times Y) \rightarrow C^{\infty}(\{ 0 \} \times Y)$ and its spectrum is given by

\begin{eqnarray}    \label{E:4.26}
\Spec \left( Q_{D}(0) \right) & = & \left\{ \frac{1}{L} \right\} \cup \left\{ \sqrt{\mu_{k}} \left( 1 + \frac{2}{e^{2L \sqrt{\mu_{k}}} - 1} \right) \mid 0 < \mu_{k} \in \Spec \left( \Delta_{Y} \right) \right\} ,
\end{eqnarray}

\noindent
where the multiplicity of $\frac{1}{L}$ is $q_{0}$.
This shows that

\begin{eqnarray}     \label{E:4.27}
& & \ln \Det \Delta_{W, N, D}  \\
& = & q_{0} \ln 2 - 2L (\ln 2 - 1) \cdot \Res_{s=-\frac{1}{2}} \zeta_{\Delta_{Y}}(s) +
L \cdot \Fp_{s=-\frac{1}{2}} \zeta_{\Delta_{Y}}(s) + \sum_{\mu_{j} > 0} \ln \left( 1 + e^{-2L\sqrt{\mu_{j}}} \right).  \nonumber
\end{eqnarray}

\noindent
{\it Remark} : $Q_{D}(0)$'s in (\ref{E:4.22}) and (\ref{E:4.26}) are different operators. In fact, $Q_{D}(0)$ in (\ref{E:4.22}) is defined on
$C^{\infty}(\{ 0 \} \times Y) \oplus C^{\infty}(\{ L \} \times Y)$  and $Q_{D}(0)$ in (\ref{E:4.26}) is defined on
$C^{\infty}(\{ 0 \} \times Y)$.

\vspace{0.2 cm}
Using Corollary \ref{Corollary:4.5} again with $Z = \{ 0 \} \times Y$ imposed by the Neumann condition and (\ref{E:4.27}), we obtain the following equality.

\begin{eqnarray}   \label{E:4.28}
\ln \Det \Delta_{W, N, {\frak R}_{\alpha}} & = & \ln \Det \Delta_{W, N, D} - {\frak s}_{\alpha} + \ln \Det (Q_{D}(0) + \alpha),
\end{eqnarray}

\noindent
where $Q_{D}(0) + \alpha$ is defined on $C^{\infty}(\{ L \} \times Y)$ and its spectrum is

\begin{eqnarray}   \label{E:4.29}
\Spec \left( Q_{D}(0) + \alpha \right) & = & \left\{ \sqrt{\mu_{k}} + \alpha - \frac{2 \sqrt{\mu_{k}}}{e^{2L \sqrt{\mu_{k}}} + 1}  \mid \mu_{k} \in \Spec \left( \Delta_{Y} \right) \right\}.
\end{eqnarray}

\noindent
When $- \alpha \notin \Spec(Q_{D}(0)) \cup \Spec(\sqrt{\Delta_{Y}})$,
the zeta-determinant $\ln \Det \Delta_{W, N, {\frak R}_{\alpha}}$ is given by

\begin{eqnarray}   \label{E:4.30}
\ln \Det \Delta_{W, N, {\frak R}_{\alpha}} & = & - {\frak s}_{\alpha} +
q_{0} \ln 2 + \ln \Det \big( \sqrt{\Delta_{Y}} + \alpha \big)  - 2L (\ln 2 - 1) \cdot \Res_{s=-\frac{1}{2}} \zeta_{\Delta_{Y}}(s)  \\
& & + ~ L \cdot \Fp_{s=-\frac{1}{2}} \zeta_{\Delta_{Y}}(s)
+ \sum_{\mu_{j} > 0} \ln \bigg( 1 - \frac{\sqrt{\mu_{j}} - \alpha}{(\sqrt{\mu_{j}} + \alpha)e^{2L\sqrt{\mu_{j}}}}  \bigg) \qquad (\Mod ~ 2 \pi i). \nonumber
\end{eqnarray}

\vspace{0.2 cm}

In the remaining part of this section, we give an example showing Theorem \ref{Theorem:3.1} and Theorem \ref{Theorem:2.16}.
Let $M = [0, L] \times Y$ with the product metric as before.
We impose the Neumann boundary condition on $\partial M = \{ 0, L \} \times Y$.
For $0 < a < L$, we cut $M$ into two pieces $M_{1} = [0, a] \times Y$ and $M_{2} = [a, L] \times Y$
with ${\mathcal N} = \{ a \} \times Y$.
For simple computation, we assume that $\sqrt{\Delta_{Y}}$ is non-negative and
 $\alpha \notin \Spec(\pm \sqrt{\Delta_{Y}})$ so that both
$\sqrt{\Delta_{Y}} + \alpha$ and $\sqrt{\Delta_{Y}} - \alpha$ are invertible operators.
Let ${\frak R}^{\alpha}_{1} = \nabla_{\partial_{\mathcal N}} + \alpha \Id$ and ${\frak R}^{\alpha}_{2} = \nabla_{\partial_{\mathcal N}} + \alpha \Id$, where $\partial_{\mathcal N}$ is defined in Section 2.
Using the fact that ${\frak R}_{2}^{\alpha} = - {\frak R}_{1}^{-\alpha}$ with (\ref{E:4.24}) and (\ref{E:4.30}), it follows that

\begin{eqnarray}    \label{E:4.31}
& & \ln \Det^{\ast} \Delta_{M, N, N} - \ln \Det \Delta_{M_{1}, N, {\frak R}_{1}^{\alpha}} -
\ln \Det \Delta_{M_{2}, {\frak R}_{2}^{\alpha}, N}   \\
& = & ({\frak s}_{\alpha} + {\frak s}_{-\alpha}) + q_{0} \ln \frac{L}{2} + \frac{1}{2} \ln \Det^{\ast} \Delta_{Y} -
\ln \Det ( \sqrt{\Delta_{Y}} + \alpha) - \ln \Det ( \sqrt{\Delta_{Y}} - \alpha)    \nonumber \\
& & + ~ \sum_{\mu_{j} > 0} \ln \left( \frac{1 - e^{-2L \sqrt{\mu_{j}}}}{\left( 1 - \frac{\sqrt{\mu_{j}} - \alpha}{(\sqrt{\mu_{j}} + \alpha) e^{2a\sqrt{\mu_{j}}}} \right)\left( 1 - \frac{\sqrt{\mu_{j}} + \alpha}{(\sqrt{\mu_{j}} - \alpha) e^{2(L-a)\sqrt{\mu_{j}}}} \right)} \right)  \qquad (\Mod ~ 2 \pi i), \nonumber
\end{eqnarray}

\noindent
where ${\frak s}_{\alpha} + {\frak s}_{-\alpha}$ is given in (\ref{E:2.61}).
We next consider the right hand side of Theorem \ref{Theorem:3.1}.
Let $\{ w_{1}, \cdots, w_{q_{0}} \}$ be an orthonormal basis of $\Ker \Delta_{Y} = \Ker Q_{1}(0) = \Ker Q_{2}(0) = \Ker R_{S}(0)$, where
$Q_{0}(0)$ and $Q_{2}(0)$ are defined on $C^{\infty}(\{ a \} \times Y)$.
Then, $\{ \frac{1}{\sqrt{L}} w_{1}, \cdots, \frac{1}{\sqrt{L}} w_{q_{0}} \}$ is an orthonormal basis of $\Ker \Delta_{W, N, N}$.
Since

\begin{eqnarray}  \label{E:4.32}
(Q_{1}(0) + \alpha)^{-1} w_{j} & = & \frac{1}{\alpha} w_{j}, \qquad
\langle \frac{1}{\sqrt{L}} w_{i}, \frac{1}{\sqrt{L}} w_{j} \rangle_{Y} ~ = ~ \frac{1}{L} \delta_{ij}, \qquad
\end{eqnarray}

\noindent
the matrices ${\mathcal A}$ in (\ref{E:3.5}) and ${\mathcal C}$ in (\ref{E:3.8}) are given by

\begin{eqnarray}   \label{E:4.33}
{\mathcal A} & = & \frac{1}{\alpha} \Id_{q_{0} \times q_{0}}, \qquad  {\mathcal C} ~ = ~ \frac{1}{L} \Id_{q_{0} \times q_{0}},
\end{eqnarray}

\noindent
which shows that

\begin{eqnarray}   \label{E:4.34}
\ln \ddet {\mathcal A} \overline{{\mathcal A}^{T}} & = & - q_{0} \ln \alpha^{2}, \qquad
\ln \ddet {\mathcal C} ~ = ~ - q_{0} \ln L.
\end{eqnarray}

\noindent
Simple computation shows that

\begin{eqnarray}   \label{E:4.35}
\Spec \left( Q_{1}(0) + \alpha \right) & = & \left\{ \sqrt{\mu_{j}} + \alpha - \frac{2 \sqrt{\mu_{j}}}{e^{2a \sqrt{\mu_{j}}} + 1} \mid
\mu_{j} \in \Spec(\Delta_{Y}) \right\}, \\
\Spec \left( Q_{2}(0) - \alpha \right) & = & \left\{ \sqrt{\mu_{j}} - \alpha - \frac{2 \sqrt{\mu_{j}}}{e^{2 (L - a) \sqrt{\mu_{j}}} + 1} \mid
\mu_{j} \in \Spec(\Delta_{Y}) \right\},  \nonumber
\end{eqnarray}

\noindent
which together with Lemma \ref{Lemma:2.7} shows that the spectrum of $R_{S}(0)$ is given as follows.

\begin{eqnarray}   \label{E:4.36}
& & \hspace{5.0 cm} \Spec \left( R_{S}(0) \right) ~ = ~ \\
& & \left\{ ~ \frac{2 \sqrt{\mu_{j}}}{\mu_{j} - \alpha^{2}} ~ \frac{1 - e^{- 2L \sqrt{\mu_{j}}}}{\left( 1 - \frac{\sqrt{\mu_{j}} - \alpha}{(\sqrt{\mu_{j}} + \alpha) e^{2a\sqrt{\mu_{j}}}}\right)\left( 1 - \frac{\sqrt{\mu_{j}} + \alpha}{(\sqrt{\mu_{j}} - \alpha) e^{2(L-a)\sqrt{\mu_{j}}}}\right) } ~ \mid ~ \mu_{j} \in \Spec(\Delta_{Y}) \right\}.   \nonumber
\end{eqnarray}

\noindent
Hence, the zeta-determinant of $R_{S}(0)$ is given as follows.

\begin{eqnarray}   \label{E:4.37}
& & \hspace{3.0 cm} \ln \Det^{\ast} R_{S}(0)  \\
& = & \ln \Det^{\ast} \left( ~ \frac{2 \sqrt{\Delta_{Y}}}{\Delta_{Y} - \alpha^{2}} \right) ~ + ~ \sum_{\mu_{j} > 0} \ln \left( \frac{1 - e^{- 2L \sqrt{\mu_{j}}}}{\left( 1 - \frac{\sqrt{\mu_{j}} - \alpha}{(\sqrt{\mu_{j}} + \alpha) e^{2a\sqrt{\mu_{j}}}}\right)\left( 1 - \frac{\sqrt{\mu_{j}} + \alpha}{(\sqrt{\mu_{j}} - \alpha) e^{2(L-a)\sqrt{\mu_{j}}}}\right)}  \right) \quad (\Mod ~ 2 \pi i).     \nonumber
\end{eqnarray}

\noindent
We note that

\begin{eqnarray}   \label{E:4.100}
\ln \Det^{\ast} \left( ~ \frac{2 \sqrt{\Delta_{Y}}}{\Delta_{Y} - \alpha^{2}} \right) & = &
\ln 2 \cdot \zeta_{\frac{\sqrt{\Delta_{Y}}}{\Delta_{Y} - \alpha^{2}}}(0)
~ + ~ \ln \Det^{\ast} \frac{\sqrt{\Delta_{Y}}}{\Delta_{Y} - \alpha^{2}}  \\
& = & \ln 2 \cdot \zeta_{\big(\sqrt{\Delta_{Y}} - \alpha^{2} \sqrt{\Delta_{Y}}^{-1} \big)}(0)
~ - ~ \ln \Det^{\ast} \left(\sqrt{\Delta_{Y}} - \alpha^{2} \sqrt{\Delta_{Y}}^{-1} \right) \quad (\Mod ~ 2 \pi i),   \nonumber
\end{eqnarray}

\noindent
where $\zeta_{\big(\sqrt{\Delta_{Y}} - \alpha^{2} \sqrt{\Delta_{Y}}^{-1} \big)}(s)$ is defined for non-zero eigenvalues of $\Delta_{Y}$.
It follows from (\ref{E:2.50}) and (\ref{E:2.56}) that

\begin{eqnarray}   \label{E:4.38}
& & \zeta_{\left( \sqrt{\Delta_{Y}} - \alpha^{2} \sqrt{\Delta_{Y}}^{-1} \right)}(0) ~ = ~
\big( {\frak a}_{\frac{m-1}{2}} - q_{0} \big) +
2 \sum_{k=1}^{[\frac{m-1}{2}]} \frac{1}{k !} \cdot {\frak a}_{\frac{m-1}{2} - k} \cdot \alpha^{2k},  \\
& & \ln \Det^{\ast} \left( \sqrt{\Delta_{Y}} - \alpha^{2} \sqrt{\Delta_{Y}}^{-1} \right) ~ = ~  \ln \Det (\sqrt{\Delta_{Y}} + \alpha )
+ \ln \Det (\sqrt{\Delta_{Y}} - \alpha ) - q_{0} \ln (- \alpha^{2})  \nonumber \\
& & \hspace{1.5 cm}  - \frac{1}{2} \ln \Det^{\ast} \Delta_{Y}
+ ~ 2 \sum_{k=1}^{[\frac{m-1}{2}]} \frac{1}{k !} \cdot {\frak a}_{\frac{m-1}{2} - k} \cdot \alpha^{2k} \cdot
\left( \sum_{p=1}^{2k-1} \frac{1}{p} - \sum_{p=1}^{k-1} \frac{1}{p} \right) \quad (\Mod ~ 2 \pi i),   \nonumber
\end{eqnarray}

\noindent
where ${\frak a}_{\frac{m-1}{2} - k}$ is understood to be zero if $m-1$ is odd.
Considering $a_{0}$ given in Theorem \ref{Theorem:2.16},
the right hand side of Theorem \ref{Theorem:3.1} is given, up to $\Mod 2 \pi i$,  as follows.

\begin{eqnarray}   \label{E:4.39}
& & a_{0} +  \ln (-1)^{q_{0}}  - \ln \ddet {\mathcal C} + \ln \ddet {\mathcal A} \overline{{\mathcal A}^{T}}
+ \ln \Det^{\ast} R_{S}(0)  \\
& = & a_{0} + \ln (-1)^{q_{0}} + q_{0} \ln L - q_{0} \ln \alpha^{2}
~ + ~ \ln 2 \cdot \bigg( \big( {\frak a}_{\frac{m-1}{2}} - q_{0} \big) + 2 \sum_{k=1}^{[\frac{m-1}{2}]} \frac{{\frak a}_{\frac{m-1}{2} - k} ~ \alpha^{2k}}{k !} \bigg)   \nonumber \\
& & - \ln \Det (\sqrt{\Delta_{Y}} + \alpha ) - \ln \Det (\sqrt{\Delta_{Y}} - \alpha ) + q_{0} \ln (- \alpha^{2}) +
\frac{1}{2} \ln \Det^{\ast} \Delta_{Y}  \nonumber \\
& & - 2 \sum_{k=1}^{[\frac{m-1}{2}]} \frac{\alpha^{2k} {\frak a}_{\frac{m-1}{2} - k}}{k !}
\left( \sum_{p=1}^{2k-1} \frac{1}{p} - \sum_{p=1}^{k-1} \frac{1}{p} \right)
 +  \sum_{\mu_{j} > 0} \ln \left( \frac{1 - e^{- 2L \sqrt{\mu_{j}}}}{\left( 1 - \frac{\sqrt{\mu_{j}} - \alpha}{(\sqrt{\mu_{j}} + \alpha) e^{2a\sqrt{\mu_{j}}}}\right)\left( 1 - \frac{\sqrt{\mu_{j}} + \alpha}{(\sqrt{\mu_{j}} - \alpha) e^{2(L-a)\sqrt{\mu_{j}}}}\right)}  \right)   \nonumber \\
& = & - q_{0} \ln 2 + q_{0} \ln L - \ln \Det (\sqrt{\Delta_{Y}} + \alpha ) - \ln \Det (\sqrt{\Delta_{Y}} - \alpha ) + \frac{1}{2} \ln \Det^{\ast} \Delta_{Y}   \nonumber \\
& & - \sum_{k=1}^{[\frac{m-1}{2}]} \frac{\alpha^{2k} {\frak a}_{\frac{m-1}{2} - k}}{k !}
\left( 2 \sum_{p=1}^{2k-1} \frac{1}{p} - \sum_{p=1}^{k-1} \frac{1}{p} \right)
 +  \sum_{\mu_{j} > 0} \ln \left( \frac{1 - e^{- 2L \sqrt{\mu_{j}}}}{\left( 1 - \frac{\sqrt{\mu_{j}} - \alpha}{(\sqrt{\mu_{j}} + \alpha) e^{2a\sqrt{\mu_{j}}}}\right)\left( 1 - \frac{\sqrt{\mu_{j}} + \alpha}{(\sqrt{\mu_{j}} - \alpha) e^{2(L-a)\sqrt{\mu_{j}}}}\right)}  \right).   \nonumber
\end{eqnarray}

\noindent
This agrees with (\ref{E:4.31}).

\vspace{0.3 cm}

\section{Neumann Boundary Condition}

We keep the same notations and assumptions as in Theorem \ref{Theorem:2.11}.
If $\partial M \neq \emptyset$, we impose the Dirichlet or Neumann boundary condition $B$ in $\partial M$ as before.
When $S = 0$, then ${\frak R}_{i}$'s in Definition \ref{Definition:2.1} are reduced to the Neumann boundary condition
$N$.
In this case, to emphasize the Neumann boundary condition and $S = 0$, we write $R_{\Neu}(\lambda)$ and ${\mathcal W}(\lambda)$ rather than $R_{S}(\lambda)$, ${\mathcal W}_{S}(\lambda)$
so that by Lemma \ref{Lemma:2.7} $R_{\Neu}(\lambda)$ is given by

\begin{eqnarray}   \label{E:5.1}
R_{\Neu}(\lambda) & = & \delta_{d} \cdot {\mathcal W}(\lambda)^{-1} \cdot \iota.
\end{eqnarray}

\noindent
Theorem \ref{Theorem:2.11} is rewritten as follows.

\begin{eqnarray}   \label{E:5.2}
 \ln \Det \left( \Delta_{M, B} + \lambda \right) - \ln \Det \left( \Delta_{M_{0}, N, N, B} + \lambda \right)
& = & \sum_{j=0}^{[\frac{m-1}{2}]} a_{j} \lambda^{j} + \ln \Det R_{\Neu}(\lambda),
\end{eqnarray}

\noindent
where $- a_{0}$ is the constant term in the asymptotic expansion of $\ln \Det R_{\Neu}(\lambda)$ for $|\lambda| \rightarrow \infty$.
In fact, $a_{0}$ is computed by the formula given in (\ref{E:2.34}). Hence, if $\Dim M$ is even, then $a_{0} = 0$. If $M$ has a product structure on a collar neighborhood of ${\mathcal N}$, then, $a_{0}  =  - \ln 2 \cdot \left( \zeta_{\Delta_{{\mathcal N}}}(0) + \Dim \Ker \Delta_{{\mathcal N}} \right)$ by Theorem \ref{Theorem:2.16}.

We next discuss (\ref{E:5.2}) when $\lambda \rightarrow 0^{+}$. If all three operators in (\ref{E:5.2}) are invertible, we can simply put $\lambda = 0$ in (\ref{E:5.2}).
However, if $\Delta_{M_{0}, N, N, B}$ is not invertible, $R_{\Neu}(0)$ is not well defined.
To avoid this difficulty, we are going to make the following assumption.

\vspace{0.2 cm}
\noindent
{\bf Assumption A} We assume that $M_{0}$ has at least two components $M_{1}$ and $M_{2}$ so that $M_{0} = M_{1} \cup M_{2}$ such that

\begin{eqnarray*}
\Ker Q_{1}(0) ~ = ~ \Ker Q_{2}(0) ~ = ~ \Ker R_{\DN}(0),
\end{eqnarray*}

\noindent
where in this case $V_{1}(\lambda) = V_{2}(\lambda) = 0$.

\vspace{0.2 cm}
\noindent
{\it Remark}:
(1) If $M$ is a closed manifold and $\Delta_{M}$ is the scalar Laplacian, then the Assumption A is satisfied since the kernel of $Q_{i}(0)$ consists of constant functions.

\noindent
(2) If $M_{1} = M_{2}$ and $M$ is a double of $M_{1}$,
{\it i.e.} $M = M_{1} \cup_{Y} M_{1}$, then $\Delta_{M}$ satisfies the Assumption A since $Q_{1}(0) = Q_{2}(0)$.

\vspace{0.2 cm}
\noindent
Under the Assumption A, we define $Q_{i}(0)^{-1}$ and $R_{\DN}(0)^{-1}$ on the orthogonal complement of $\Ker Q_{i}(0)$ so that
$R_{\Neu}(0)^{-1} = Q_{1}(0) R_{\DN}(0)^{-1} Q_{2}(0)$ is well defined.

\vspace{0.2 cm}

In the remaining part of this section, we assume the Assumption A and let $\ell_{0} = \Dim \Ker Q_{1}(0)$. Then,

\begin{eqnarray}   \label{E:5.3}
\Dim \Ker \Delta_{M, B} ~ = ~ \Dim \Ker \Delta_{M_{1},N, B}  ~ = ~ \Dim \Ker \Delta_{M_{2},N, B} ~ = ~ \ell_{0}.
\end{eqnarray}

\noindent
Hence, we have

\begin{eqnarray}   \label{E:5.4}
& & \ln \Det \left( \Delta_{M, B} + \lambda \right) ~ - ~ \ln \Det \left( \Delta_{M_{1}, N, B} + \lambda \right) ~ - ~ \ln \Det \left( \Delta_{M_{2}, N, B} + \lambda \right)   \\
& = & \left\{  \ell_{0} \ln \lambda ~ + ~ \ln \Det^{\ast} \Delta_{M, B} ~ + ~ O(\lambda) \right\} ~ - ~
\left\{  2 \ell_{0} \ln \lambda ~ + ~ \ln \Det^{\ast} \Delta_{M_{1}, N, B} ~ + ~ \ln \Det^{\ast} \Delta_{M_{2}, N, B} ~ + ~ O(\lambda) \right\}  \nonumber \\
& = & -  \ell_{0} \ln \lambda ~ + ~ \left\{ \ln \Det^{\ast} \Delta_{M, B} ~ - ~  \ln \Det^{\ast} \Delta_{M_{1}, N, B} ~ - ~ \ln \Det^{\ast} \Delta_{M_{2}, N, B} ~ \right\} ~ + ~ O(\lambda).   \nonumber
\end{eqnarray}

We denote the holomorphic families of eigenvalues and corresponding eigensections of $R_{\Neu}(\lambda)^{-1}$ by
$\{ \rho_{j}(\lambda) \mid j = 1, 2, \cdots \}$ and $\{ \varphi_{j}(\lambda) \mid j = 1, 2, \cdots \}$, where
$\rho_{1}(0) = \cdots = \rho_{\ell_{0}}(0) = 0$ and $\{ \varphi_{1}(0), \cdots, \varphi_{\ell_{0}}(0) \}$ is an orthonormal basis of
$\Ker R_{\Neu}(0)^{-1}$.
Then, we get

\begin{eqnarray}   \label{E:5.5}
\ln \Det R_{\Neu}(\lambda) & = & - \ln \rho_{1}(\lambda) \cdots \rho_{\ell_{0}}(\lambda)
~ + ~ \ln \Det^{\ast} R_{\Neu}(0) ~ + ~ O(\lambda).
\end{eqnarray}

\noindent
Similarly, we denote by $\{ \tau_{j}(\lambda) \mid j \in {\mathbb N} \}$, $\{ \omega_{j}(\lambda) \mid j \in {\mathbb N} \}$
the holomorphic families of eigenvalues and corresponding eigensections of $R_{\DN}(\lambda)$ and by
$\{ \eta^{(i)}_{j}(\lambda) \mid j \in {\mathbb N} \}$, $\{ \phi^{(i)}_{j}(\lambda) \mid j \in {\mathbb N} \}$
the eigenvalues and eigensections of $Q_{i}(\lambda)$, where
$\tau_{1}(0) = \cdots = \tau_{\ell_{0}}(0) = 0$, $\eta^{(i)}_{1}(0) = \cdots \eta^{(i)}_{\ell_{0}}(0) = 0$, and
$\{ \omega_{1}(0), \cdots, \omega_{\ell_{0}}(0) \}$, $\{\phi^{(i)}_{1}, \cdots, \phi^{(i)}_{\ell_{0}} \}$ are orthonormal bases of $\Ker R_{\DN}(0) = \Ker Q_{i}(0)$.
Hence, $\Span \{ \varphi_{1}(0), \cdots, \varphi_{\ell_{0}}(0) \}~ = \Span~~\{ \omega_{1}(0), \cdots, \omega_{\ell_{0}}(0) \} ~ = ~
\Span \{\phi^{(i)}_{1}, \cdots, \phi^{(i)}_{\ell_{0}} \}$.
We define $\ell_{0} \times \ell_{0}$ unitary matrices ${\mathcal K}^{(i)} = \left( {\mathcal K}_{jk}^{(i)} \right)$,
${\mathcal L}^{(i)} = \left( {\mathcal L}_{jk}^{(i)} \right)$ and ${\mathcal M}^{(i)} = \left( {\mathcal M}_{jk}^{(i)} \right)$ as follows.

\begin{eqnarray}    \label{E:5.6}
\varphi_{j}(0) ~ = ~ \sum_{k=1}^{\ell_{0}} {\mathcal K}_{jk}^{(i)} \phi_{k}^{(i)}(0), \qquad
\phi_{j}^{(i)}(0) ~ = ~ \sum_{k=1}^{\ell_{0}} {\mathcal L}_{jk}^{(i)} \omega_{k}(0), \qquad
\omega_{j}(0) ~ = ~ \sum_{k=1}^{\ell_{0}} {\mathcal M}_{jk}^{(i)} \phi_{k}^{(i)}(0).
\end{eqnarray}

\noindent
It follows from the definition that

\begin{eqnarray}    \label{E:5.7}
{\mathcal K}^{(1)} {\mathcal L}^{(1)} = {\mathcal K}^{(2)} {\mathcal L}^{(2)}, \qquad
{\mathcal L}^{(1)} {\mathcal M}^{(1)} =  {\mathcal L}^{(2)} {\mathcal M}^{(2)} = \Id.
\end{eqnarray}

\noindent
Since $\rho_{j}(\lambda)$ is holomorphic with
$\rho_{j}(0) = 0$ for $1 \leq j \leq \ell_{0}$, we have $\lim_{\lambda \rightarrow 0} \frac{\rho_{j}(\lambda)}{\lambda} = \rho_{j}^{\prime}(0)$.
The same property holds for $\tau_{j}(\lambda)$, $\eta^{(i)}_{j}(\lambda)$.
For two one parameter families $f(\lambda)$ and $g(\lambda)$ of sections, we define an equivalence relation "$\approx$" as follows.

\begin{eqnarray*}
f(\lambda) \approx g(\lambda) \quad \text{if and only if} \quad
\lim_{\lambda \rightarrow 0} \left( f(\lambda) - g(\lambda) \right) = 0.
\end{eqnarray*}

\noindent
For $1 \leq j, k \leq \ell_{0}$, we have

\begin{eqnarray}   \label{E:5.8}
 \frac{\rho_{j}(\lambda)}{\lambda} \delta_{jk} & = & \frac{1}{\lambda}
\bigg\langle Q_{1}(\lambda) ~ R_{\DN}(\lambda)^{-1} ~ Q_{2}(\lambda) \varphi_{j}(\lambda), ~ \varphi_{k}(0) \bigg\rangle  \\
& = & \frac{1}{\lambda}
\bigg\langle Q_{1}(\lambda) ~ R_{\DN}(\lambda)^{-1} ~ Q_{2}(\lambda) \left( \sum_{a=1}^{\ell_{0}} {\mathcal K}_{ja}^{(2)} \phi^{(2)}_{a}(\lambda) +  \varphi_{j}(\lambda) - \sum_{a=1}^{\ell_{0}} {\mathcal K}_{ja}^{(2)} \phi^{(2)}_{a}(\lambda) \right), ~ \varphi_{k}(0) \bigg\rangle  \nonumber \\
& = & \sum_{a=1}^{\ell_{0}} {\mathcal K}_{ja}^{(2)} \frac{\eta^{(2)}_{a}(\lambda)}{\lambda} ~ \bigg\langle Q_{1}(\lambda) ~ R_{\DN}(\lambda)^{-1} ~ \phi^{(2)}_{a}(\lambda), ~ \varphi_{k}(0) \bigg\rangle ~ + ~ C(\lambda),  \nonumber
\end{eqnarray}

\noindent
where for $\lambda \rightarrow 0$

\begin{eqnarray}    \label{E:5.9}
\lim_{\lambda \rightarrow 0} C(\lambda) & = & \lim_{\lambda \rightarrow 0} \frac{1}{\lambda} ~
\bigg\langle R_{\Neu}(\lambda)^{-1}
\left( \varphi_{j}(\lambda) - \sum_{a=1}^{\ell_{0}} {\mathcal K}_{ja}^{(2)} \phi^{(2)}_{a}(\lambda) \right), ~ \varphi_{k}(0) \bigg\rangle \\
& = & \lim_{\lambda \rightarrow 0} ~
\bigg\langle \varphi_{j}(\lambda) - \sum_{a=1}^{\ell_{0}} {\mathcal K}_{ja}^{(2)} \phi^{(2)}_{a}(\lambda),
~ \frac{R_{\Neu}(\overline{\lambda})^{-1} - R_{\Neu}(0)^{-1}}{\overline{\lambda}} ~ \varphi_{k}(0) \bigg\rangle   \nonumber \\
& = & \bigg\langle \varphi_{j}(0) - \sum_{a=1}^{\ell_{0}} {\mathcal K}_{ja}^{(2)} \phi^{(2)}_{a}(0),
~ \bigg(\frac{d}{d\lambda}\big|_{\lambda=0} R_{\Neu}(\lambda)^{-1} \bigg) ~ \varphi_{k}(0) \bigg\rangle
 ~ = ~  0.    \nonumber
\end{eqnarray}

\noindent
Using the similar method, we have

\begin{eqnarray}    \label{E:5.10}
 \frac{\rho_{j}(\lambda)}{\lambda} \delta_{jk} & \approx & \sum_{a=1}^{\ell_{0}} {\mathcal K}_{ja}^{(2)} \frac{\eta^{(2)}_{a}(\lambda)}{\lambda} ~ \bigg\langle Q_{1}(\lambda) ~ R_{\DN}(\lambda)^{-1} ~ \phi^{(2)}_{a}(\lambda), ~ \varphi_{k}(0) \bigg\rangle,   \\
& = & \sum_{a=1}^{\ell_{0}} {\mathcal K}_{ja}^{(2)} \frac{\eta^{(2)}_{a}(\lambda)}{\lambda} ~ \bigg\langle Q_{1}(\lambda) ~ R_{\DN}(\lambda)^{-1} ~ \left( \sum_{b=1}^{\ell_{0}} {\mathcal L}_{ab}^{(2)} \omega_{b}(\lambda) + \phi^{(2)}_{a}(\lambda) - \sum_{b=1}^{\ell_{0}} {\mathcal L}_{ab}^{(2)} \omega_{b}(\lambda) \right), ~ \varphi_{k}(0) \bigg\rangle,  \nonumber  \\
& \approx & \sum_{a, b=1}^{\ell_{0}} {\mathcal K}_{ja}^{(2)} \frac{\eta^{(2)}_{a}(\lambda)}{\lambda} ~ {\mathcal L}_{ab}^{(2)} \frac{1}{\tau_{b}(\lambda)} ~
\bigg\langle Q_{1}(\lambda) ~ \omega_{b}(\lambda), ~ ~ \varphi_{k}(0) \bigg\rangle     \nonumber\\
& \approx & \sum_{a, b=1}^{\ell_{0}} {\mathcal K}_{ja}^{(2)} \frac{\eta^{(2)}_{a}(\lambda)}{\lambda} ~ {\mathcal L}_{ab}^{(2)} \frac{1}{\tau_{b}(\lambda)} ~
\bigg\langle Q_{1}(\lambda) ~ \left( \sum_{c=1}^{\ell_{0}} {\mathcal M}^{(1)}_{bc} \phi^{(1)}_{c}(\lambda) +  \omega_{b}(\lambda) - \sum_{c=1}^{\ell_{0}} {\mathcal M}^{(1)}_{bc} \phi^{(1)}_{c}(\lambda)  \right), ~ ~ \varphi_{k}(0) \bigg\rangle     \nonumber \\
& \approx & \sum_{a, b, c=1}^{\ell_{0}} \left( {\mathcal K}_{ja}^{(2)} \frac{\eta^{(2)}_{a}(\lambda)}{\lambda} \right) ~ \left( {\mathcal L}_{ab}^{(2)} \frac{\lambda}{\tau_{b}(\lambda)} \right) ~
\left( {\mathcal M}^{(1)}_{bc} \frac{\eta_{c}^{(1)}(\lambda)}{\lambda} \right)
~ \overline{{\mathcal K}^{(1)}_{kc}},     \nonumber
\end{eqnarray}

\noindent
where we used that $\lim_{\lambda \rightarrow 0} \frac{\lambda}{\tau_{b}(\lambda)}$ exists and nonzero by
(\ref{E:3.9}).
We denote by $~\{ \psi^{(i)}_{1}, \cdots, \psi^{(i)}_{\ell_{0}} \}~$ an orthonormal basis of $~\Ker \Delta_{M_{i}, N}~$ for $i = 1, 2$.
For an orthonormal basis
$~\{ \phi_{1}^{(i)}(0), \cdots, \phi_{\ell_{0}}^{(i)}(0) \}~$ of $\Ker  Q_{i}(0)$, we choose sections $\Phi_{1}^{(i)}(0), \cdots, \Phi_{\ell_{0}}^{(i)}(0) $ on $M_{i}$ such that

\begin{eqnarray}    \label{E:5.11}
\Delta_{M_{i}} \Phi_{j}^{(i)}(0) = 0, \qquad \Phi_{j}^{(i)}(0)\big|_{{\mathcal N}} = \phi_{j}^{(i)}(0), \qquad  B(\Phi_{j}^{(i)}(0)) = 0.
\end{eqnarray}

\noindent
We define matrices ${\mathcal R}^{(i)} = \left( {\mathcal R}^{(i)}_{ij} \right)$ and ${\mathcal S}^{(i)} = \left( {\mathcal S}^{(i)}_{ij} \right)$
as follows:

\begin{eqnarray}    \label{E:5.12}
{\mathcal R}^{(i)}_{jk} ~ = ~ \big\langle \Phi_{j}^{(i)}(0), ~ \psi_{k}^{(i)} \big\rangle_{M_{i}}, \qquad
{\mathcal S}^{(i)}_{jk} ~ = ~ \big\langle \psi_{j}^{(i)}|_{{\mathcal N}}, ~ \psi_{k}^{(i)}|_{{\mathcal N}} \big\rangle_{{\mathcal N}}.
\end{eqnarray}

\noindent
It is not difficult to see that

\begin{eqnarray}    \label{E:5.13}
\ddet \big( {\mathcal R}^{(i)} \overline{{\mathcal R}^{(i) T}} \big) & = & \frac{1}{\ddet {\mathcal S}^{(i)}}.
\end{eqnarray}

\noindent
Similar computation as (\ref{E:4.12}) shows (cf. \cite{KL5}) that

\begin{eqnarray}    \label{E:5.14}
\lim_{\lambda \rightarrow 0} \frac{\eta_{j}^{(i)}(\lambda)}{\lambda} \delta_{jk} ~ = ~
\langle \Phi_{j}^{(i)}(0), ~ \Phi_{k}^{(i)}(0) \rangle  ~ = ~
\left( {\mathcal R}^{(i)} \overline{{\mathcal R}^{(i) T}} \right)_{jk}.
\end{eqnarray}

\noindent
Since $\lim_{\lambda \rightarrow 0} \frac{\lambda}{\tau_{b}(\lambda)} ~ = ~
(\overline{{\mathcal B}^{T}}{\mathcal B})_{bb}$ by (\ref{E:3.9}), we have

\begin{eqnarray}    \label{E:5.15}
\lim_{\lambda \rightarrow 0} \frac{\rho_{j}(\lambda)}{\lambda} \delta_{jk} & = &
\sum_{a, b, c=1}^{\ell_{0}} \left( {\mathcal K}_{ja}^{(2)} \left( {\mathcal R}^{(2)} \overline{{\mathcal R}^{(2)T}} \right)_{aa} \right) ~
\left( {\mathcal L}_{ab}^{(2)} \left(\overline{{\mathcal B}^{T}} {\mathcal B} \right)_{bb} \right) ~
\left( {\mathcal M}^{(1)}_{bc} \left( {\mathcal R}^{(1)} \overline{{\mathcal R}^{(1)T}} \right)_{cc} \right)
~ \overline{{\mathcal K}^{(1)}_{kc}}     \nonumber \\
& = & \left( {\mathcal K}^{(2)} {\mathcal R}^{(2)} \overline{{\mathcal R}^{(2)T}} ~
 {\mathcal L}^{(2)} \left({\mathcal B}^{T} {\mathcal B} \right) ~
 {\mathcal M}^{(1)}  {\mathcal R}^{(1)} \overline{{\mathcal R}^{(1)T}}~ \overline{{\mathcal K}^{(1) T}} \right)_{jk} .
\end{eqnarray}

\noindent
By (\ref{E:5.7}) this leads to

\begin{eqnarray}    \label{E:5.16}
\lim_{\lambda \rightarrow 0} \frac{\rho_{1}(\lambda) \cdots \rho_{\ell_{0}}(0)}{\lambda^{\ell_{0}}} & = &
\ddet  \left( {\mathcal K}^{(2)} {\mathcal R}^{(2)} \overline{{\mathcal R}^{(2)T}} ~
 {\mathcal L}^{(2)} \left(\overline{{\mathcal B}^{T}} {\mathcal B} \right) ~
 {\mathcal M}^{(1)}  {\mathcal R}^{(1)} \overline{{\mathcal R}^{(1)T}}~ \overline{{\mathcal K}^{(1) T}} \right)  \\
 & = &  \ddet \left( {\mathcal R}^{(2)} \overline{{\mathcal R}^{(2)T}} ~
  \left(\overline{{\mathcal B}^{T}} {\mathcal B} \right) ~ {\mathcal R}^{(1)} \overline{{\mathcal R}^{(1)T}}~ \right)
 ~ = ~ \frac{\ddet {\mathcal C}}{\ddet {\mathcal S}^{(1)} ~ \ddet {\mathcal S}^{(2)}},    \nonumber
\end{eqnarray}

\noindent
where ${\mathcal C}$ is defined in (\ref{E:3.8}).
This together with (\ref{E:5.4}) and (\ref{E:5.5}) yields the following result.

\begin{theorem}   \label{Theorem:5.1}
Let $M$ be a compact Riemannian manifold with boundary $\partial M$, where $\partial M$ may be empty.
We choose a closed hypersurface ${\mathcal N}$ such that ${\mathcal N} \cap \partial M = \emptyset$ and $M - {\mathcal N}$ has at least two components.
We assume that $\Delta_{M}$ is a Laplacian satisfying (\ref{E:2.2}) and the Assumption A. Then,
\begin{eqnarray*}
& & \ln \Det^{\ast} \Delta_{M, B} ~ - ~  \ln \Det^{\ast} \Delta_{M_{1}, N, B} ~ - ~ \ln \Det^{\ast} \Delta_{M_{2}, N, B} ~   \\
& = & a_{0} ~ - ~ \ln \ddet {\mathcal C} ~ + ~ \ln \ddet {\mathcal S}^{(1)} ~ + ~ \ln \ddet {\mathcal S}^{(2)}
~ - ~ \ln \Det^{\ast} \left( Q_{1}(0) R_{\DN}(0)^{-1} Q_{2}(0) \right),
\end{eqnarray*}
where $a_{0}$ is the constant term in the asymptotic expansion of $\ln \Det \left( Q_{1}(\lambda) R_{\DN}(\lambda)^{-1} Q_{2}(\lambda) \right)$ for $|\lambda| \rightarrow \infty$.
If $\Dim M$ is even, then $a_{0} = 0$.
If $M$ has a product structure near ${\mathcal N}$ so that $\Delta_{M}$ is
$- \partial_{x_{m}}^{2} + \Delta_{{\mathcal N}}$ on a collar neighborhood of ${\mathcal N}$, then
\begin{eqnarray*}
a_{0} & = & - \ln 2 \cdot \left( \zeta_{\Delta_{{\mathcal N}}}(0) + \Dim \Ker \Delta_{{\mathcal N}}\right).
\end{eqnarray*}
\end{theorem}

\noindent
If $\Delta_{M}$ is a scalar Laplacian on a closed manifold $M$, then it follows that $\ell_{0} = 1$ and

\begin{eqnarray}    \label{E:5.17}
{\mathcal C} ~ = ~ \frac{\vol({\mathcal N})}{\vol(M)}, \qquad
{\mathcal S}^{(1)} ~ = ~ \frac{\vol({\mathcal N})}{\vol(M_{1})}, \qquad  {\mathcal S}^{(2)} ~ = ~ \frac{\vol({\mathcal N})}{\vol(M_{2})},
\end{eqnarray}

\noindent
which leads to the following result.

\begin{corollary}   \label{Corollary:5.2}
We make the same assumptions as in Theorem \ref{Theorem:5.1}.
Let $\Delta_{M}$ be a Laplacian acting on smooth functions on a closed Riemannian manifold. Then,
\begin{eqnarray*}
& & \ln \Det^{\ast} \Delta_{M} ~ - ~  \ln \Det^{\ast} \Delta_{M_{1}, N} ~ - ~ \ln \Det^{\ast} \Delta_{M_{2}, N} ~   \\
& = & a_{0} ~ + ~ \ln \frac{\vol({\mathcal N}) \cdot \vol(M)}{\vol(M_{1}) \cdot \vol(M_{2})} ~ - ~ \ln \Det^{\ast} \left( Q_{1}(0) R_{\DN}(0)^{-1} Q_{2}(0) \right).
\end{eqnarray*}
\end{corollary}

\vspace{0.2 cm}
We next discuss an analogue of Corollary \ref{Corollary:3.3} when $\Dim M = 2$. Let $(M, g)$ be a $2$-dimensional closed Riemannian manifold and $\Delta_{M}$ be a scalar Laplacian. As before,
we consider a conformal change of a metric $g_{ij}(\epsilon) = e^{2\epsilon F} g_{ij}$, and we get
$\Delta_{M}(\epsilon) := e^{- 2 \epsilon F} \Delta_{M}$. Let $\ell({\mathcal N})$ be the length of ${\mathcal N}$.
Then, Corollary \ref{Corollary:5.2} is rewritten by

\begin{eqnarray}    \label{E:5.21}
& & \ln \frac{\Det^{\ast} Q_{1}(0) R_{\DN}(0)^{-1} Q_{2}(0)}{\ell({\mathcal N})} (\epsilon)  \\
& = & - \ln \Det^{\ast} \Delta_{M}(\epsilon) + \sum_{i=1}^{2} \ln \Det^{\ast} \Delta_{M_{i}, N}(\epsilon)
~ + ~ \ln \vol(M)(\epsilon) ~ - ~ \sum_{i=1}^{2} \ln \vol(M_{i})(\epsilon),    \nonumber
\end{eqnarray}

\noindent
where we used the fact that $a_{0}= 0$ when $\Dim M$ is even.
We denote

\begin{eqnarray}   \label{E:5.22}
\Tr \big( F e^{- t \Delta_{M}} \big) ~ \sim ~ \sum_{k=0}^{\infty} \beta_{k}(F) ~ t^{\frac{-2 + k}{2}}, \qquad
\Tr \big( F e^{- t \Delta_{M_{i}, N}} \big) ~ \sim ~ \sum_{k=0}^{\infty} \beta^{(i)}_{k}(F) ~ t^{\frac{-2 + k}{2}}.
\end{eqnarray}

\noindent
By Theorem 2.12 in \cite{BG2} and Theorem 3.3.1, Theorem 3.5.1 in \cite{Gi2} (or Chapter 4 in \cite{Ki}), it follows that

\begin{eqnarray}    \label{E:5.23}
& & \frac{d}{d\epsilon}\big|_{\epsilon=0} \ln \frac{\Det^{\ast} Q_{1}(0) R_{\DN}(0)^{-1} Q_{2}(0)}{\ell({\mathcal N})} (\epsilon) \\
& = & 2 \bigg\{ \beta_{2}(F) ~ - ~ \frac{ \int_{M} F dx }{\vol(M)} ~ - ~ \sum_{i=1}^{2} \bigg( \beta_{2}^{(i)}(F) ~ - ~ \frac{ \int_{M_{i}} F dx }{\vol(M_{i})} \bigg) \bigg\} ~ + ~ \frac{2 \int_{M} F dx}{\vol(M)} ~ - ~ \sum_{i=1}^{2}\frac{2 \int_{M_{i}} F dx}{\vol(M_{i})}    \nonumber \\
& = & \beta_{2}(F) - \beta^{(1)}_{2}(F) - \beta_{2}^{(2)}(F) ~ = ~  0,    \nonumber
\end{eqnarray}

\noindent
which leads to the following result. This is an analogue of Theorem 1.1 in \cite{GG} and \cite{EW}.

\begin{corollary}
Let $(M, g)$ be a $2$-dimensional closed Riemannian manifold and $\Delta_{M}$ be a scalar Laplacian.
We choose a closed hypersurface ${\mathcal N}$ such that the closure of $M - {\mathcal N}$ is a union of $M_{1}$ and $M_{2}$. Then, $\frac{\Det^{\ast} Q_{1}(0) R_{\DN}(0)^{-1} Q_{2}(0)}{\ell({\mathcal N})}$ is a conformal invariant.
\end{corollary}

\vspace{0.2 cm}

Finally, we give a simple example of showing Theorem \ref{Theorem:5.1}.

\vspace{0.2 cm}
\noindent
{\it Example 5.3} : Let $M = [0, ~ L] \times Y$ and ${\mathcal N} = \{ a \} \times Y$,
where $0 < a < L$. We put $M_{1} = [0, ~ a] \times Y$ and $M_{2} = [a, ~ L] \times Y$.
We denote by $\Delta_{M, N, N}$ and $\Delta_{M_{i}, N, N}$ the Laplacian $- \frac{\partial^{2}}{\partial u^{2}} + \Delta_{Y}$ with the Neumann boundary condition on $\{ 0, L\} \times Y$ and $\{ a \} \times Y$ as before.
It follows from (\ref{E:4.24}) that

\begin{eqnarray}    \label{E:5.18}
& & \ln \Det^{\ast} \Delta_{M, N, N} - \ln \Det^{\ast} \Delta_{M_{1}, N, N} - \ln \Det^{\ast} \Delta_{M_{2}, N, N} \\
 & = &
- q_{0} \ln 2 ~ + ~ q_{0} \ln \frac{L}{a(L - a)} ~ - ~ \frac{1}{2} \ln \Det^{\ast} \Delta_{Y}
~ + ~ \sum_{\mu_{j} > 0} \ln \frac{1 - e^{-2L\sqrt{\mu_{j}}}}{(1 - e^{-2a\sqrt{\mu_{j}}})(1 - e^{-2(L- a)\sqrt{\mu_{j}}})},  \nonumber
\end{eqnarray}

\noindent
where $q_{0} = \Dim \Ker \Delta_{Y}$.
We note that $Q_{1}(0)$ and $Q_{2}(0)$ defined on $\{ a \} \times Y$ satisfy the Assumption A with
$\Ker Q_{i}(0) = \Ker \Delta_{Y}$.
It follows from (\ref{E:4.37}) that

\begin{eqnarray}   \label{E:5.19}
& & - \ln \Det^{\ast} \left( Q_{1}(0) R_{\DN}(0)^{-1} Q_{2}(0) \right) ~ = ~ \ln \Det^{\ast} R_{\Neu}(0)  \\
& = & \ln 2 \cdot \zeta_{{\Delta_{Y}}}(0) -
\frac{1}{2} \ln \Det^{\ast} {\Delta_{Y}}
~ + ~ \sum_{0 < \mu_{j} \in \Spec( \Delta_{Y})}
\ln \frac{1 - e^{-2L \sqrt{\mu_{j}}}}{(1 - e^{-2a\sqrt{\mu_{j}}})(1 - e^{-2(L- a)\sqrt{\mu_{j}}})}.   \nonumber
\end{eqnarray}

\noindent
Let $\{ w_{1}, \cdots, w_{q_{0}} \}$ be an orthonormal basis of $\Ker \Delta_{Y}$. Then,
$\{ \frac{1}{\sqrt{L}} w_{1}, \cdots, \frac{1}{\sqrt{L}} w_{q_{0}} \}$ is an orthonormal basis of $\Ker \Delta_{M, N, N}$.
Hence, ${\mathcal C} = \frac{1}{L} \Id$ and $\ddet {\mathcal C}  =  \left( \frac{1}{L} \right)^{q_{0}}$. Similarly, we get $\ddet {\mathcal S}^{(1)} ~ = ~ \left( \frac{1}{a} \right)^{q_{0}}$ and
$\ddet {\mathcal S}^{(2)} ~ = ~ \left( \frac{1}{L-a} \right)^{q_{0}}$.
Since $a_{0} = - \ln 2 \cdot \left( \zeta_{\Delta_{Y}}(0) + q_{0} \right)$, it follows that

\begin{eqnarray}   \label{E:5.20}
a_{0}  -  \ln \ddet {\mathcal C}  +  \ln \ddet {\mathcal S}^{(1)}  +  \ln \ddet {\mathcal S}^{(2)} & = &
 - \ln 2 \cdot \left( \zeta_{\Delta_{Y}}(0) + q_{0} \right)
  +  q_{0} \ln \frac{L}{a(L-a)}.
\end{eqnarray}

\noindent
This agrees with Theorem \ref{Theorem:5.1}.

\vspace{1.0 cm}


\end{document}